\documentclass[12pt]{amsart}
\usepackage{latexsym, amsmath, amssymb, amsthm, amsfonts, amstext, amscd, amsxtra, amsbsy, mathrsfs, stmaryrd, mathdots, bbold}
\usepackage[all]{xy}
\usepackage{tikz}
\usepackage{soul}
\usepackage[makeroom]{cancel}
\usepackage{bbm}
\usepackage{color}
\usepackage{hyperref}
\usepackage{shorttoc}

\newtheorem{theorem}{Theorem}[section]
\newtheorem{lemma}[theorem]{Lemma}
\newtheorem{prop}[theorem]{Proposition}

\theoremstyle{definition}
\newtheorem{conj}[theorem]{Conjecture}

\newtheorem{defn}[theorem]{Definition}
\newtheorem{example}[theorem]{Example}

\newcommand{\mbf}{\mathbf}
\newcommand{\mbb}{\mathbb}
\newcommand{\mcal}{\mathcal}
\newcommand{\mscr}{\mathscr}
\newcommand{\mrm}{\mathrm}
\newcommand{\mrk}{\mathfrak}

\def\ro{\text{ro}}
\def\co{\text{co}}
\def\mf{\mathfrak}
\def\diag{\text{diag}}

\newcommand{\Ub}{\mrm U^{\mrk b}}
\newcommand{\Uc}{\mrm U^{\mrk c}}
\newcommand{\U}{\mrm U}

\title[Equivariant K-theory approach to $\imath$-quantum groups]{Equivariant K-theory approach to $\imath$-quantum groups}

\author[Z. Fan, H. Ma and H. Xiao]{Zhaobing Fan, Haitao Ma and Husileng Xiao}
\address{College of mathematics science, Harbin Engineering University, Harbin, 15001, China.}
\email{fanz@math.ksu.edu(Z. Fan) \\
 hmamath@hrbeu.edu.cn(H. Ma) \\
 hslxiao@hrbeu.edu.cn(H. Xiao)}

\date{\today}

\begin{document}

\begin{abstract}
Various constructions for quantum groups have been generalized to $\imath$-quantum groups.
Such generalization is called $\imath$-program.
In this paper, we fill one of parts in the $\imath$-program.
Namely, we provide an equivariant K-theory approach to $\imath$-quantum groups associated to the Satake diagram in \eqref{eq1},
which is the Langlands dual picture of that constructed in \cite{BKLW14},
where a geometric realization of the $\imath$-quantum group is provided by using perverse sheaves.
As an application of the main results, we prove Li's conjecture \cite{L18} for the special cases with the satake diagram in \eqref{eq1}.
\end{abstract}

\maketitle

\setcounter{tocdepth}{1}
\tableofcontents

\section{Introduction}

\subsection{Quantum symmetric pairs}
Let $\mf g$ be a Lie algebra, and $\theta: \mf g\rightarrow \mf g$ be an involution on $\mf g$.
Let $\mf g^{\theta}$ be the fixed point subalgebra of $\mf g$.
The pair $(\mf g, \mf g^{\theta})$ is called a symmetric pair, and
the quantization of their universal enveloping algebra, denoted by $(\U_q(\mf g), \U_q(\mf g^{\theta}))$, is called a quantum symmetric pair.
The algebra $\U_q(\mf g^{\theta})$ itself is called an $\imath$-quantum group.

Due to \'E. Cartan, $\imath$-quantum groups can be classified in term of Satake diagrams \cite{Ar62},
which are bicolor Dynkin diagrams with black or white vertices and a diagram involution fixing all black vertices.
Quantum symmetric pairs are systemically studied by Letzter \cite{Le02, Le99} for finite types, and by Kolb \cite{Ko14} for Kac-Moody cases.
In general, the $\imath$-quantum group, $\U_q(\mf g^{\theta})$, is a coideal subalgebra of $\U_q(\mf g)$.
In the following two cases, $\imath$-quantum groups are degenerated to the ordinary quantum groups.
The first one is when vertices in Satake diagrams are all black.
The second one is so called quasi-split case with diagonal type, see the definition in \cite{LW19}.

\subsection{$\imath$-program}
From the above point of view, $\imath$-quantum groups are highly nontrivial generalization of the ordinary quantum groups.
Generalizing various constructions for quantum groups to $\imath$-quantum groups is called $\imath$-program, which is proposed by Bao and Wang in \cite{BW18a}.
Let us recall some constructions for quantum groups and their generalization to $\imath$-quantum groups.

Canonical basis theory for quantum groups, introduced by Lusztig in \cite{Lu90, Lu91}, is a milestone in representation theory.
The generalization of canonical bases and R-matrix to $\imath$-quantum groups were given by Bao and Wang in \cite{BW18a, BW18b},
in which $\imath$-canonical bases were defined, and by Balagovic and Kolb in \cite{BK19}, respectively.

In \cite{BKLW14} (resp. \cite{FL15}, \cite{FLLLW1}), authors considered the convolution algebras on double flag varieties associated to the algebraic group $O_{2r+1}(\mbb F_q)$
(resp. $So_{2r}(\mbb F_q)$, $Sp_{2r}(\mbb F_q((t)))$),
which are $q$-Schur type algebras.
The geometric realizations of the corresponding $\imath$-quantum groups are hence obtained.
We shall call this approach BLM realization since it was firstly given by Belinson, Lusztig and MacPherson in \cite{BLM90} for the case of type A.

It has been realized by Deng, Du, Parshall, and Wang that BLM construction has an algebraic counterpart.
The construction for type A was given in \cite{DDPW}.
In \cite{FLLLW2}, a Hecke algebraic approach to $\imath$-quantum groups and the corresponding affine $q$-Schur algebras was developed,
which is the algebraic counterpart of the geometric approach studied in \cite{FLLLW1}.
The algebraic approach to other types has been studied in \cite{LL15, LW17, LL18}.

A categorification of quantum groups was given independently by Khovanov and Lauda in \cite{KL09, KL10}, and Rouguier in \cite{R08} by using KLR algebra.
 This work had been generalized to the $\imath$-quantum group $\Ub$ (denoted by $\U^{\jmath}$ therein) by Bao, Shan, Wang and Webster in \cite{BSWW}.

In \cite{Ri90}, Ringel constructed a Hall algebra associated to a Dynkin quiver and shown that it was isomorphic to the half part of the corresponding quantum group.
In \cite{B13}, Bridgeland realized the entire quantum groups by using Hall algebra of two-step complexes.
Rencently, Lu and Wang provided a Hall algebra approach to $\imath$-quantum groups for finite types in \cite{LW19}.
In \cite{L18}, Li defined $\sigma$-quiver varieties for $\imath$-quantum groups, which are analogies of Nakajima quiver varieties.

\subsection{Langlands reciprocity}

It is well-known that affine quantum groups have two different algebraic presentations.
One is called Drinfeld-Jimbo realizations, the other one is called Drinfeld new realizations.
The geometric counterpart for the case of affine type A also has two different approaches.
The first one is using perverse sheaves on double flag varieties associated to the corresponding algebraic group $GL(\mathbb F_q((t)))$,
and
the second one is using equivariant K-theory on Steinberg varieties associated to $GL(\mbb C)$.
The second construction yields a classification of irreducible finite dimensional $\U_q(\widehat{\mathfrak{sl}_n})$-module as well as their character formulas.
Ginzburg, Reshetikhin and  Vassrot further provided a geometric realization of Schur-Weyl duality of affine type A in \cite{GRV94} by using equivariant K-theory.
For Hecke algebras,  Iwahori and Matsumoto gave a BLM-type construction of the affine Hecke algebra associated to algebraic group $G$.
The equivariant K theory construction of affine Hecke algebra was first done by Kazhdan and Lusztig \cite{Lu85, KL85, KL87}, and then improved by Ginzburg \cite{G87}.
 Namely,  there exists an equivariant K-theory approach to the affine Hecke algebra by using the Steinberg variety associated with ${}^L\!G$, where ${}^L\!G$ is the Langlands dual of $G$.
This suggests that a mysterious link exists between the BLM realization and the equivatiant K-theory realization.
Such relationship is so called Langlands reciprocity for affine quantum group of type $A_n$ in \cite{GV93}.
In light of  $\imath$-program, the BLM realization of the $\imath$-quantum group $\U^{\widetilde{\mathfrak c}}$ (denoted by $\U^{\jmath\jmath}$ therein),
and the equivariant K-theoretic approach to affine quantum groups,
they strongly suggest the existence of equivaraint K-theoretic approach to the $\imath$-quantum group $\U^{\widetilde{\mathfrak c}}$,
which is expected by Wang and Li earlier.
This is the original motivation of this project.

\subsection{Main results}
Since it is hard to find the Drinfeld new realization of $\U^{\widetilde{\mathfrak c}}$,
we shall consider a toy case in this paper, the $\imath$-quantum groups $\Uc$ and $\Ub$, whose
satake diagram is as follows.
\begin{equation}\label{eq1}
\begin{tikzpicture}[baseline = 0, scale =1.3]
		\node at (-1.5,0) {$\circ$};
		\draw (-1.4,0) to (-1.1,0);
		\node at (-1,0) {$\circ$};
		\draw (-0.9,0) to (-0.7,0);
        \draw[dashed] (-0.6,0) to (0.6,0);
		\draw (0.7,0) to (0.9,0);
		\node at (1.0,0) {$\circ$};
		\draw (1.1,0) to (1.4,0);
		\node at (1.5,0) {$\circ$};
		\draw[<->] (-1.4, 0.1) to[out=45, in=180] (-1.0, 0.4) to (1.0, 0.4) to[out=0, in=135] (1.4, 0.1);
		\draw[<->] (-0.9, 0.1) to[out=45, in=180] (-0.65, 0.2) to (0.65, 0.2) to[out=0, in=135] (0.9, 0.1);
	\end{tikzpicture}
\end{equation}

 This paper devote to study $\imath$-quantum group $\Uc$ in term of equivariant K-groups on Steinberg variety of type C, including its generators, relations and the coideal structure.
More precisely, we have the following theorem.

{\bf Theorem}[Theorems \ref{main theorem C} and \ref{thm-coideal}]
Let $\mcal{Z}^{\mf c}$ (resp. $\mcal{Z}^{\mf a}$) be the Steinberg variety of type C (resp. type A).
Then we have the following two algebra homomorphisms.
\begin{itemize}
 \item[(a)] $\Ub \longrightarrow K^{\mathbb{C}^{*} \times Sp_d(\mbb C)}(\mcal{Z}^{\mf c})_{\mathbb{K}}$.

 \item[(b)]$\Delta:\mathbb{C}_\alpha \otimes K^{ \mathbb{C}^{*}\times Sp_d(\mbb C)}(\mcal{Z}^{\mf c}) \longrightarrow \mathbb{C}_\alpha \otimes K^{\mathbb{C}^{*} \times GL_{d_1}  } (\mcal{Z}_{d_1}^{\mf a})\otimes   K^{\mathbb{C}^{*} \times Sp_{d_2}(\mbb C)}(\mcal{Z}_{d_{2}}^{\mf c})$,\\
 where $d_1+d_2=d$ and $\mbb C_{\alpha}$ is a 1-dimensional representation depending on a semisimple element $\alpha\in \mbb C^*\times Sp_d(\mbb C)$, see details in Section \ref{coideal C}.
\end{itemize}

Although $\Uc$ is isomorphic to $\Ub$ in  \cite{BKLW14},
we note that $K^{\mathbb{C}^{*} \times G_{\mf b}}(\mcal{Z}^{\mf b})_{\mathbb{K}}$ is not isomorphic to $K^{\mathbb{C}^{*} \times G}(\mcal{Z}^{\mf c})_{\mathbb{K}}$ in general.
From this point of view, $\Ub$ has two different approaches.
Although these two constructions are quite similar, the polynomial representations are different,
see Theorem \ref{polynomial rep C} and Proposition \ref{representation prop B}.

In \cite{L18}, Li constructs $\sigma$-quiver variety $\mathfrak M(\mbf v, \mbf w)^{\sigma}$,
which is the fixed point subvariety of Nakajima quiver variety by $\sigma$ and
conjectures that there is an algebra homomorphism from the universal enveloping algebra, $\U(\mathfrak g^{\theta})$, of the fixed point Lie subalgebra of $\mathfrak g$
to the top degree Broel-Moore homology $H_{{\rm top}}(\mathfrak Z)$ if $\mathfrak g$ is of finite types, where $\mathfrak Z$ is the Steinberg type $\sigma$-quiver variety.
See \cite{L18} for more details.
As an application of the main results, we prove Conjecture 5.3.4 in \cite{L18} for the special cases with the  Satake diagram in \eqref{eq1}.

\subsection{Organizations}
 This paper is organized as follows.
 In  Section 2,
 we calculate equivariant K-group on the generalized Steinberg varieties of type C,
 and its module structure on $K^{\mbb C^* \times G}(\mcal M)$,
 where $\mcal M$ is the cotangent bundle of flag varieties of type C.
 In Section 3,
 we briefly review the presentation of $\imath$-quantum group $\Ub$ and its coideal structure.
 We also prove the main results of the paper, namely, there is an algebra homomorphism from $\Ub$ to the convolution algebra arising from generalized Steinberg varieties of type C [Theorem \ref{main theorem C}].
Moreover, we provide an K-theoretic approach to the coideal structure of $\Ub$ by using the bivariant localization theorem.
  In Section 4.
  we review the presentation of $\imath$-quantum group $\Uc$ and prove similar results as that for $\Ub$ in Section 3.
  As an application, we prove Li's conjecture for the cases we consider in this paper.
 The proof of Proposition \ref{representation prop} is combinatorially involved, so
we move the detailed proof of it to the appendix.\vspace{7pt}

\noindent\textbf{Acknowledgements.}
We thank
C. Lai, L. Luo, Y. Li, and W. Wang for helpful discussions.
The first author thanks University of Virgina for the support and hospitality.
Z. Fan is partially supported by the NSF of China grant 11671108, the NSF of Heilongjiang
Province grant LC2017001 and the Fundamental Research Funds for the central universities GK2110260131.
H. Ma is partially supported by the NSF of China grant 11571119.
H. Xiao is partially supported by the NSF of China grant 11801113.

\section{Convolution algebra in equivariant K-theory}

\subsection{Equivariant K-theory}\label{sec2.1}

Let $G$ be a linear algebraic group.
Given a $G$-variety $X$,
let $Coh^G(X)$ be the category of $G$-equivariant coherent sheaves on $X$,
and $K^G(X)$ be the complexified Grothendieck group of $Coh^G(X)$.
For any coherent sheaf $\mscr F\in Coh^G(X)$, denote by $[\mscr F]$ its equivalent class in $K^G(X)$.

In the special case that $X= \{pt\}$, a $G$-equivariant sheaf on $X$ is just a $G$-vector space.
Hence $Coh^G(pt) = Rep(G)$, the category of representations of $G$, and $K^G(pt) = R(G)$, the complexified representation ring of $G$.
Moreover, for any $G$-variety $X$, the external tensor product defines a $R(G)$-module structure on $K^G(X)$.

Given three smooth $G$-varieties $M_1,~M_2,~M_3$, let
$$p_{ij}:M_1 \times M_2 \times M_3 \rightarrow M_i\times M_j$$
be the obvious projection maps.
Let $Z_{12}\subseteq M_1 \times M_2 $ and $Z_{23}\subseteq M_2 \times M_3 $ be $G$-stable closed subvarieties.
We denote
 \begin{equation}
   \label{circ}
   Z_{12}\circ Z_{23}= p_{13}(p_{12}^{-1}(Z_{12})\cap p_{23}^{-1}(Z_{23})).
 \end{equation}
If the restriction of $p_{13}$ to $p_{12}^{-1}(Z_{12})\cap p_{23}^{-1}(Z_{23})$ is a proper map,
then we define the convolution product as follows.
\begin{equation*}
\begin{split}
\star: \ \ K^{G}(Z_{12})\otimes K^{G}(Z_{23})&\longrightarrow K^{G}(Z_{12} \circ Z_{23}), \\
   [\mscr{F}_{1}] \otimes [\mscr{F}_{2}] &\mapsto Rp_{{13}{\ast}}(p_{12}^{\ast}[\mscr{F}_{1}] \overset{L}{\otimes} p_{23}^{\ast}[\mscr{F}_{2}]), \ \ \
\end{split}
\end{equation*}
where $\overset{L}{\otimes}$ is the derived tensor product.
We note that if the properness condition doesn't hold, then derived direct image $Rp_{13\ast}$ is not well-defined in general.

In the case that $M_1 =M_2 =M_3 = M$, the convolution product defines an algebra structure on $ K^G(M\times M)$.

\subsection{Flag variety of type C}

Let $V =\mbb C^{2d}$ with a nondegenerate skew-symmetric bilinear form $( , )$, and $G = Sp(V)$.
To simplify notations, we set once for all
$$N= 2n+1.$$
 Let
$$\Lambda_{\mf c}=\{ \mathbf{v}=(v_i) \in \mbb N^{N} \mid   v_i = v_{N+1-i},\quad \textstyle \sum_{i=1}^{N} v_i = 2d \}.$$
For any $W\subseteq V$, let $W^{\perp} = \{x \in V| (x,y)=0, \ \forall y\in W\}$.
For any $\mbf v \in \Lambda_{\mf c}$, we further set
$$\mathcal{F}_{\mathbf{v}}=\{ F=( 0=V_0 \subset V_1 \subset \cdots \subset V_{N}=V)\ \mid \  V_i=V_{N-i}^{\perp} ,\ \text{dim}( V_{i}/V_{i-1})= v_i,\  \forall i \}. $$
The componentwise action of $G$ on $\mcal F_{\mbf v}$ is transitive.
For any $F \in \mathcal{F}_{\mathbf{v}}$, let $P_{\mathbf{v}} $ be the stabilizer of $F$ in $G$.
Therefore, we have $G/P_{\mathbf{v}}\simeq \mcal F_{\mbf v}$.
Let $\mathcal{F} = \sqcup _{\mathbf{v} \in \Lambda_{\mf c}} \mathcal{F}_{\mathbf{v}}$.
Group $G$ naturally acts on $\mcal F$.

\subsection{Steinberg variety and the convolution algebra}
 \label{steinberg variety of type c}

Let $\mcal M = T^*\mathcal F $, the cotangent bundle of $\mathcal F$.
 More precisely, $\mcal M$ can be written into
\begin{equation*}
  \label{tangent}
  \mcal M = \{(F,x) \in \mathcal F \times \mathfrak{sp}_{2d}\mid x(F_i) \subseteq F_{i-1}, ~ \forall i \}.
\end{equation*}
The conjugate action of $G$ on fibers induces a $G$-action on $\mcal M$.
 we define a $\mbb C^* \times G$-action on $\mcal M$ by
$$ (z,g) \cdot  (F,x) = (gF, z^{-2}gxg^{-1}), \quad \forall (z,g) \in \mbb C^* \times G.$$
Let $\mcal N$ be the nilpotent variety of $\mathfrak{sp}_{2d}$.
By the definition of $\mcal M$,
 we have $\mcal M \subseteq \mcal F \times \mcal N$.
 Let $\pi: \mcal M\rightarrow \mcal N, (F,x) \mapsto x$ be the second projection map.
 It is clear that $\pi$ is a proper map.
   Moreover, one can define a $\mathbb{C}^{*} \times G$-action on $\mcal N$ such that $\pi$ is a $\mathbb{C}^{*} \times G$-equivariant map.

    Let $\mcal Z=\mcal M\times_{\mcal N} \mcal M$ be the generalized Steinberg variety of type C,
    which is  the union of conormal bundles of $G$-orbits in $\mathcal{F} \times \mathcal{F}$.
    If we want to emphasis the type,
      we shall use $\mcal Z^{\mf c}$ instead of $\mcal Z$.
 Group $\mathbb{C}^{*} \times G $ acts on $\mcal Z$ diagonally.
 Since $R(\mbb C^*) = \mbb C[q,q^{-1}]$ and $\mcal Z \subseteq \mcal M \times \mcal M$,
  by Section \ref{sec2.1}, $(K^{\mathbb{C}^{*} \times G}(\mcal Z), \star)$ is a $\mathbb{C}[q,q^{-1}]$ algebra with unit.

\subsection{Generators of the convolution algebra}
\label{Generators of type C}
In this section, we shall study generators of the algebra $K^{\mathbb{C}^{*} \times G}(\mcal Z)$.
Let us recall some notations from \cite{BKLW14}.
Let
$$\Xi_d=\{A=(a_{ij}) \in {\rm Mat}_{N\times N}(\mbb N)\ | \ \sum_{i,j}a_{ij}=2d,\ a_{ij}=a_{N+1-i,N+1-j},\ \forall \  i,\  j\}.$$
It has been shown in \cite[Section 6]{BKLW14} that the set $\Xi_d$  parameterizes $G$-orbits of $\mathcal{F} \times \mathcal{F}$.
For each $ A=(a_{ij}) \in \Xi_d$, denote
\begin{equation*}
  \ro(A) = (\sum_j a_{ij})_{i=1, 2,\cdots, N}\ {\rm and}\  \co(A) = (\sum_i a_{ij})_{j=1,2, \cdots, N}.
\end{equation*}
Moreover, for any $\mathbf{v}, \mathbf{w} \in \Lambda_{\mf c}$, the set $\Xi_d(\mathbf{v},\mathbf{w})$ parameterizes $G$-orbits of $\mathcal{F}_{\mathbf{v}}\times \mathcal{F}_{\mathbf{w}}$, where
$$\Xi_d(\mathbf{v},\mathbf{w}) = \{ M \in \Xi_d \mid \ro(M) = \mbf v,\  \co(M) = \mbf w \}.$$

Let $W_{\mf c} = \mathbb{Z}_2^{d}\ltimes S_{d}$ be the Weyl group of type $C_d$, which
has a natural action on set $\{1,2, \cdots,2d\}$.
  For  $\mathbf{v} \in \Lambda_{\mf c}$, we set $\bar{v}_{i}=\sum_{r=1}^{i} v_r $ and $[\mathbf{v}]_i=[1+\bar{v}_{i-1},\bar{v}_{i}]\subseteq \mbb N$.
   Then $[\mathbf{v}] =([\mathbf{v}]_1,[\mathbf{v}]_2,\cdots,[\mathbf{v}]_N)$ forms a partition of the set $\{1,2,\cdots,2d\}$.
    To a matrix $A \in \Xi_d$, we also associate a partition of the set $\{1,2,\cdots,2d\}$ as follows
    \begin{equation}\label{partA}
      [A]=([A]_{11}, \cdots [A]_{1N}, [A]_{21},\cdots, [A]_{NN}),
    \end{equation}
where $[A]_{ij} =[\sum\limits_{(h,k)< (i,j)}a_{hk}+1 , \sum\limits_{(h,k)<(i,j)}a_{hk}+a_{ij}]\subseteq \mbb N$,
and $<$ is the  left lexicographical order, i.e.,
$$(h,k) < (i,j) \Leftrightarrow h<i  ~\text{or}~ (h = i ~\text{and}~  k<j).$$
For any partition $[\mbf v]=([\mathbf{v}]_1,[\mathbf{v}]_2,\cdots,[\mathbf{v}]_N)$ of the set $\{1,2,\cdots,2d\}$,
we define a partition of the set $\{1,2,\cdots,d\}$ as follows,
\begin{equation*}
  [\mbf v]^{\mf c} = ([\mathbf{v}]_1,[\mathbf{v}]_2, \cdots,[\mbf v]_n, [\mathbf{v}]^{\mf c}_{n+1}) \ {\rm and}\ [\mathbf{v}]^{\mf c}_{n+1}= [\bar v_n+1, d].
\end{equation*}
It is clear that $|[\mathbf{v}]^{\mf c}_{n+1}| = \frac{1}{2}|[\mathbf{v}]_{n+1}|$,
where $|-|$ is the cardinality of the set.
Similarly, for any partition $[A]$ of the set $\{1,2,\cdots,2d\}$ in \eqref{partA},
we define a partition, denoted by $[A]^{\mf c}$, of the set $\{1,2,\cdots,d\}$,
\begin{equation*}
  [A]^{\mf c} = ([A]_{11}, \cdots [A]_{1N}, [A]_{21},\cdots, [A]_{n+1,n}, [A]_{n+1,n+1}^{\mf c}),
\end{equation*}
where $[A]^{\mf c}_{n+1,n+1}= [\sum_{(h,k)< (n+1,n+1)}a_{hk}+1, d]$.
To a partition $[\mbf v]$, we associate a subgroup $W_{[\mbf v]^{\mf c}}$ of   $W_{\mf c}$,
$$W_{[\mbf v]^{\mf c}} = S_{[\mbf v]_1}\times S_{[\mbf v]_2} \times \cdots \times S_{[\mbf v]_n} \times (\mathbb{Z}_2^{|[\mbf v]_{n+1}^\mbf c|}\ltimes S_{[\mbf v]_{n+1}^\mbf c}),$$
where  $S_{[\mbf v]_i}$ is the subgroup of $S_d$ consisting of all permutations which preserve $[\mbf v]_i$.

It is well known that there is also a bijection between the double coset $W_{[\mathbf{v}]^{\mf c}}\backslash W_{\mf c}/W_{[\mathbf{w}]^{\mf c}}$ and the $G$-orbits on $ \mathcal{F}_{\mathbf{v}}\times \mathcal{F}_{\mathbf{w}}$.
 The precise correspondence is given by
$$ W_{[\mathbf{v}]}\backslash W_{\mf c}/W_{[\mathbf{w}]} \longrightarrow \Xi_d(\mathbf{v},\mathbf{w}), \quad  W_{[\mathbf{v}]}\sigma W_{[\mathbf{w}]} \mapsto \mathbf{m}^{\sigma}=(m^{\sigma}_{ij}). $$
Here $m^{\sigma}_{ij}=\#\{ a \in [\mathbf{v}]_i \mid  \sigma (a)  \in [\mathbf{w}]_j \}$.

We can define an order $\preceq$ on $\Xi_d$ as follows.
For any $A=(a_{ij}), B=(b_{ij})\in \Xi_d$, $A \preceq B $ if and only if
\begin{align}
\ro(A) = \ro(B),\ \co(A)=\co(B),\ {\rm and}\ \sum_{r\leq i; s\geq j} a_{rs} \leq \sum_{r\leq i; s\geq j} b_{rs}, \  \forall  i<j.
\end{align}
This order is compatible with the Bruhat order on $W_{\mf c}$ via the above bijection.
By Lemma 3.8 in \cite{BKLW14}, we have the following proposition.
\begin{prop}
\label{orbit structure}
{\rm (a)}  For any $A,B \in \Xi_d$, $\mathcal{O}_A \subseteq  \overline{\mathcal{O}}_B$ if $A \preceq B$.

{\rm (b)} For any $A,B \in \Xi_d$,
let $M(A,B)$ be the set consisting of all matrices $C\in \Xi_d$ such that $\mcal O_C \subseteq p_{13}(p_{12}^{-1}\mathcal{O}_A \cap p_{23}^{-1}\mathcal{O}_B)$.
We have
$$p_{13}(p_{12}^{-1}\mathcal{O}_A \cap p_{23}^{-1}\mathcal{O}_A)= \bigcup\limits_{C \in M(A,B)}\mathcal{O}_C.$$
Moreover, there exists a unique matrix, denoted by $A\circ B $, such that for any $C \in M(A,B)$, we have $C \preceq A\circ B$.

{\rm (c)} For any $A,A',B,B' \in \Xi_d$ with $A' \preceq A $, $B' \preceq B$, we have
$$\quad C' \preceq A\circ B,~  \forall ~ C' \in M(A',B'). $$

\end{prop}

Let $\mathbf{R} = \mathbb{C}[x_1^{\pm 1},x_2^{\pm 1},\cdots,x_d^{\pm 1}]$.
We shall define a natural action of Weyl group $W_{\mf c}$ on $\mbf R$ as follows,
which induces a well-defined action of $W_{[\mbf v]^{\mf c}}$ (resp. $W_{[A]^{\mf c}}$) on $\mbf R$.
 For any permutation $\sigma \in S_{d}$, the action of $\sigma$ on $\mbf R$ is given by
$$\sigma: \mathbf{R} \rightarrow \mathbf{R},~~f(x_1^{\pm 1},x_2^{\pm 1},\cdots,x_d^{\pm 1}) \mapsto  f(x_{\sigma(1)}^{\pm 1},x_{\sigma(2)}^{\pm 1},\cdots,x_{\sigma(d)}^{\pm 1}).$$
 We now define a $\mathbb{Z}_2^d$-action on $\mathbf{R}$.
 We shall denote by $[{\bf 1}]_m = (\delta_{im})_{i = 1,2,\cdots, d}$ the $m$-th generator of $\mbb Z_2^d$, $\forall m=1,\cdots,d$.
 The action of  $[{\bf 1}]_m $ on  $\mathbf{R}$ is defined by
\begin{equation*}
  \begin{split}
[{\bf 1}]_{m}: \mathbf{R} &\rightarrow \mathbf{R},\\
 f(x_1^{\pm 1},\cdots,x_{m-1}^{\pm 1}, x_m^{\pm 1},x_{m+1}^{\pm 1},\cdots x_d^{\pm 1}) &\mapsto  f(x_1^{\pm 1},\cdots,x_{m-1}^{\pm 1}, x_m^{\mp 1},x_{m+1}^{\pm 1},\cdots x_d^{\pm 1}).
  \end{split}
\end{equation*}
For any partition $I = (I_1,I_2, \cdots, I_{2r+1})$ of $2d$,
denote by   $\mathbf{R}^{W_{I^{\mf c}}}$ the subring of $\mathbf{R}$ fixed by $W_{I^{\mf c}}$.
If  $J = (J_1,J_2, \cdots, J_{2r+1})$ is another partition with $W_{J^{\mf c}} \subseteq W_{I^{\mf c}}$.
We define a map
 $$W_{I^{\mf c}}/W_{J^{\mf c}} : \mathbf{R}^{W_{J^{\mf c}}} \rightarrow \mathbf{R}^{W_{I^{\mf c}}}, \quad f \mapsto \sum\limits_{\sigma \in W_{I^{\mf c}}/W_{J^{\mf c}}} \sigma(f).$$

\begin{prop}\label{pushforwad C}
For any $\mathbf{v}, \mathbf{v}_1, \mathbf{v}_2 \in \Lambda_{\mf c}$, $A \in \Xi_d(\mathbf{v}_1,\mathbf{v}_2)$.

{\rm (a)} There exist  $\mathbb{C}$-algebra isomorphisms $K^{G}(\mathcal{F}_{\mathbf{v}}) \cong \mathbf{R}^{W_{[\mathbf{v}]^{\mf c}}} \text{and} \ K^{G}(\mathcal{O}_A) \cong \mathbf{R}^{W_{[A]^{\mf c}}}$.

{\rm (b)} The first projection map $p_{1,A} : \mathcal{O}_A \rightarrow \mathcal{F}_{\mathbf{v}_1}$ is a smooth fibration.
Moreover, if $\mathcal{O}_A $ is closed, then the direct image morphism $R_{p_{1,A} *}$ is given by
 \begin{equation*}
\begin{split}
 &R_{p_{1,A} *}(f) \\
  &= W_{[\mathbf{v}_1]^{\mf c}}/W_{[A]^{\mf c}}(f\prod_{(s,t)}(1-x_s/x_t)^{-1}\prod_{(l,m)}(1-x_l/x_m)^{-1}(1-x_lx_m)^{-1}\prod_{(p,q)}(1-x_px_q)^{-1}),
 \end{split}
 \end{equation*}
where the product runs over the following ranges.
$$(s,t) \in  [A]_{ja} \times [A]_{jb},~ 1 \leq a < b \leq 2n+1,~ j \in [1,n];$$
$$(l,m) \in [A]_{n+1,a} \times ([A]_{n+1,b} \cup  [A]_{n+1,n+1}^{\mf c}),\ a,b\in [1,n] ~\text{and}~ a< b; $$
$$(p,q) \in  [A]_{n+1,a} \times [A]_{n+1,a},~ p \leq q, 1 \leq a \leq n.$$
\end{prop}
\begin{proof}

Fix $A \in \Xi_d(\mathbf{v}, \mbf w)$, $(F,F') \in \mathcal{O}_A$. There exists a decomposition  $V = \oplus_{1 \leq i,j \leq N}Z_{ij}$ such that
$$F_a = \oplus_{i \leq a}Z_{ij}~\text{and}~ F_b' = \oplus_{j \leq b}Z_{ij}, ~\forall \ a,b \in [1,N].$$
 Let $P_F$ and $P_{F'}$ be the stabilizer of $F$ and $F'$ in G, respectively.
Since $\mathcal{O}_A = \{g\cdot(F,F')\ |\ g\in G\}$,
 the induction property \cite[Section 5.2.16]{CG} implies that
 $$K^G(\mathcal{O}_A) \cong K^G(G/(P_F \cap P_{F'}))\cong K^{P_F \cap P_{F'}}(pt) \cong R(P_F \cap P_{F'}),$$
 where $R(P_F \cap P_{F'})$ is the representation ring of $P_F \cap P_{F'}$.
 Moreover, $Sp(Z_{n+1,n+1})\times \prod_{(i,j) < (n+1,n+1)}GL(Z_{ij})$ is isomorphic to the Levi subgroup of $P_F \cap P_{F'}$.
 By  \cite[Theorem 6.1.4]{CG}  again, we have
$$ K^{G}(\mathcal{O}_A) \cong \mathbf{R}^{W_{[A]^{\mf c}}}.$$
In the special case that $A =\text{diag}( \mathbf v)$, the above result implies that
$$K^{G}(\mathcal{F}_{\mathbf{v}}) \cong \mathbf{R}^{W_{[\mathbf{v}]^{\mf c}}}.$$
Statement (a) follows.

The first part of (b) is clear. We now prove the second part of (b).
Suppose that $\mathcal{O}_A$ is closed.
The fiber of the projection $p_{1,A}$ is isomorphic to $P_{F}/(P_F \cap P_{F'})$.
Thus  the class $T_A^* \in \mathbf{R}$ of  the relative cotangent bundle $T_{p_{1,A}}^*$ at the origin is
\begin{equation}\label{cotangent formula}
T_A^* = \sum_{(s,t)}x_s/x_t\sum_{(l,m)}(x_l/x_m+ x_lx_m)\sum_{(p,q)}x_px_q,
\end{equation}
where the sums run over the following ranges,
 $$(s,t) \in  [A]_{ja} \times [A]_{jb},~ 1 \leq a < b \leq N,~ j \in [1,n];$$
$$(l,m) \in [A]_{n+1,a} \times ([A]_{n+1,b} \cup [A]_{n+1,n+1}^{\mf c}) , a,b\in [1,n] ~\text{and}~ a<b; $$
$$(p,q) \in  [A]_{n+1,a} \times [A]_{n+1,a},~ p \leq q, 1 \leq a \leq n.$$
By Lefschetz formula, for any sheaf $\mathscr{F} \in K^G(\mathcal{O}_A)$, we have
$$R_{p_{1,A} *}[\mathscr{F}] =  W_{[\mathbf{v}_1]^{\mf c}}/W_{[A]^{\mf c}}([\mathscr{F}] \otimes \bigwedge(T_A^*)^{-1}),$$
where $\bigwedge(T_A^*) = \sum_i(-1)^i\bigwedge^iT_A^*$. Statement (b) follows.
\end{proof}

Let $\mathbb{A} = \mathbb{C}[q,q^{-1}]$ and  $\mathbb{K} = \mathbb{C}(q)$.
For any vector space $V$ over $\mathbb{C}$, we set $V_{\mathbb{A}} = V \otimes \mathbb{A}$ and $V_{\mathbb{K}} = V \otimes \mathbb{K}$.
For any $A \in \Xi_d$,  let $ \mcal{Z}_{A}$ be the conormal bundle of $\mathcal{O}_A$,
and $\mcal{Z}_{\preceq A} = \cup_{B \preceq A} \mcal{Z}_{B}$.
Proposition \ref{orbit structure} implies that $K^{\mathbb{C}^{*}\times G}(\mcal{Z})$ has a filtration structure indexed by $\Xi_d$ as follows.
$$K^{\mathbb{C}^{*}\times G}(\mcal{Z}_{\preceq A}) \star K^{\mathbb{C}^{*}\times G}(\mcal{Z}_{\preceq B}) \subseteq K^{\mathbb{C}^{*}\times G} (\mcal{Z}_{\preceq A\circ B}).$$
On the other hand,  by the cellular fibration lemma \cite[Lemma 5.5.1]{CG}, there is a canonical short exact sequence
$$0\rightarrow K^{\mathbb{C}^{*}\times G}(\mcal{Z}_{\prec A})\rightarrow K^{\mathbb{C}^{*}\times G}(\mcal{Z}_{\preceq A}) \rightarrow K^{\mathbb{C}^{*}\times G}(\mcal{Z}_{ A})\rightarrow 0,$$
where $\mcal{Z}_{\prec A} = \bigcup _{B \preceq A, B\neq A} \mcal Z_{\preceq B}$.
By Proposition \ref{pushforwad C}(a) and Thom isomorphism \cite[Theorem 5.4.17]{CG}, we have the following $\mathbb{A}$-module isomorphism
\begin{equation}
\label{eq2}
\oplus_{A\in \Xi_d} K^{\mathbb{C}^{*}\times G}(\mcal{Z}_{\preceq A})/K^{\mathbb{C}^{*}\times G}(\mcal{Z}_{\prec A}) \cong \oplus_{A\in \Xi_d}K^{\mathbb{C}^{*}\times G}(\mcal{Z}_{ A}) \cong  \oplus_{A\in \Xi_d} \mathbf{R}^{W_{[A]^{\mbf c}}}_{\mathbb{A}}.
\end{equation}
By Proposition \ref{orbit structure}, for any $A' \preceq A, B' \preceq B$ and $(A',B') \neq (A,B)$, we have
$$\mcal{Z}_{\preceq A'} \circ \mcal{Z}_{\preceq B'} \subseteq \mcal{Z}_{ \prec A\circ B}.$$
By \eqref{eq2},
 the convolution product on $ K^{\mathbb{C}^{*}\times G}(\mcal{Z})$ induces a well-defined algebra structure on $\oplus_{A\in \Xi_d} \mathbf{R}^{W_{[A]^{\mbf c}}}_{\mathbb{A}}$,
  denoted by $\overline{\star}$.

\begin{example}\label{fgenerator}

In the case that $N = 3$, let
$$A_1=\left(
  \begin{array}{ccc}
     v_1 & 0 &  0  \\
    a & v_2+2b & a  \\\
 0 &  0 & v_1 \\
  \end{array}
\right), \quad A_2=\left(
  \begin{array}{ccc}
     v_1+a & 0 & 0  \\
    b & v_2 & b   \\
  0 &  0 & v_1+a \\
  \end{array}
\right),
$$
$$ A_3=\left(
  \begin{array}{ccc}
     v_1 & 0 & 0  \\
    b+a & v_2 & b+a   \\
  0 &  0 & v_1 \\
  \end{array}
\right).$$
It is obvious that $A_1\circ A_2 = A_3.$
Set
\begin{equation*}
  \begin{split}
I_1 &= [1,v_1],\\
I_2 &= [v_1+1,v_1+a],\\
I_3 &= [v_1+a+1,v_1+a+b],\\
I_4 &=[v_1+a+b+1,v_1+a+b+v_2],\\
I_5 &=[v_1+a+b+v_2+1,v_1+v_2+a + 2b],\\
I_6 &=[v_1+v_2+a + 2b+1,v_1+v_2+2a + 2b],\\
I_7 &=[v_1+v_2+2a + 2b+1, 2d].
  \end{split}
\end{equation*}
Thus
\begin{equation*}
  \begin{split}
W_{[A_1]^{\mf c}} &= W_{(I_1,I_2,I_3\cup I_4 \cup I_5, I_6,I_7)^{\mf c}}= S_{I_1}\times S_{I_2} \times  (\mathbb{Z}_2^{|(I_3\cup I_4 \cup I_5)^{\mbf c}|}\ltimes S_{(I_3\cup I_4 \cup I_5)^{\mbf c}}),\\
W_{[A_2]^{\mf c}} &= W_{(I_1\cup I_2,I_3,I_4 , I_5,I_6\cup I_7)^{\mf c}}= S_{I_1 \cup I_2}\times S_{I_3} \times  (\mathbb{Z}_2^{| I_4^{\mbf c} |}\ltimes S_{I_4^{\mbf c}} ),\\
W_{[A_3]^{\mf c}} &= W_{(I_1,I_2\cup I_3,I_4 , I_5\cup I_6, I_7)^{\mf c}}= S_{I_1 }\times S_{I_2\cup I_3} \times (\mathbb{Z}_2^{| I_4^{\mbf c} |}\ltimes S_{I_4^{\mbf c}} ).
 \end{split}
\end{equation*}
By Proposition \ref{pushforwad C}(b), we have
 \begin{equation*}
   \begin{split}
\bar \star:  \mathbf{R}^{W_{[A_1]^{\mf c}}} \otimes \mathbf{R}^{W_{[A_2]^{\mf c}}} &\longrightarrow \mathbf{R}^{W_{[A_1 \circ A_2]^{\mf c}}},\\
 f\otimes g &\mapsto W_{[A_3]^{\mf c}}/(W_{[A_1]^{\mf c}} \cap W_{[A_2]^{\mf c}})(fg \prod\limits_{i \in I_2}\prod\limits_{j \in I_3} \frac{1 - q^2x_j/x_i}{1 - x_i/x_j}),
   \end{split}
 \end{equation*}
where $W_{[A_3]^{\mf c}}/(W_{[A_1]^{\mf c}} \cap W_{[A_2]^{\mf c}}) = S_{I_2\cup I_3}/(S_{I_2}\times S_{I_3})$.
For any $f \in \mathbf{R}^{S_{I_2 \cup I_3}}$, we have
$$S_{I_2\cup I_3}/(S_{I_2}\times S_{I_3})(f\prod\limits_{i \in I_2}\prod\limits_{j \in I_3} \frac{1 - q^2x_j/x_i}{1 - x_i/x_j})=P(q)f,$$
where $P(q)$ is a polynomial in $q$ which can be computed explicitly as follows.
$$P(q) = q^{ab}\frac{[a+b]!}{[a]![b]!}, \quad [k]! = \prod_{i=1}^k[i],\ {\rm and}\ \  [k] = \frac{q^k - q^{-k}}{q-q^{-1}}.$$
 Since $P(q)$ is invertible in $\mbb K$, the map
  $\bar \star:  \mathbf{R}^{W_{[A_1]^{\mf c}}}_{\mathbb{K}} \otimes \mathbf{R}^{W_{[A_2]^{\mf c}}}_{\mathbb{K}} \longrightarrow \mathbf{R}^{W_{[A_1 \circ A_2]^{\mf c}}}_{\mathbb{K}}$
  is surjective.
\end{example}

\begin{prop} \label{diagonal action}
 Let $A$ be a diagonal matrix.
 For any $B \in \Xi_d$ with $\co(A) =\ro(B)$,
 we have
$$f \bar \star g= fg, \quad \forall f\in \mathbf{R}^{W_{[A]^{\mf c}}},\   g\in \mathbf{R}^{W_{[B]^{\mf c}}}.$$
\end{prop}
\begin{proof}
Since $A$ is a diagonal matrix, the projection map $p_{13}: p_{12}^{-1}\mathcal{O}_A \cap p_{23}^{-1}\mathcal{O}_B \longrightarrow \mathcal{O}_B$ is an isomorphism.
 By Corollary 3 in \cite{V98}, the proposition follows.
\end{proof}

Let $E_{ij} = (a_{kl})_{N\times N}$ such that $a_{kl} = 1$ if $ (k,l) = (i,j)$ and $ 0,$ otherwise.
Define $E_{ij}^{\theta} \triangleq  E_{ij} + E_{N+1-i,N+1-j}$ and $E_{ij}^{\theta}(\mathbf{v},a)\triangleq \diag(\mathbf{v}) + a E_{ij}^{\theta}$, where  $\mathbf{v} = (v_1, \cdots, v_N)$ such that $v_i = v_{N+1 -i}$ and $\sum\limits_{i=1}^N v_i= 2d-2a$.

\begin{prop}\label{inducpropo}
Fix $\mathbf{v} \in \Lambda_{\mf c}, A,B \in \Xi_d$  with $\co(A) = \ro(B)$.
\begin{itemize}
\item[(a)]Let $A=E_{h,h+1}^{\theta}(\mathbf{v},a)$, and $l=\max\{i\mid b_{h+1,i}\neq 0\}$.
If $b_{h+1,l} \geq a$,
then
$$A\circ B=B+a(E_{hl}^\theta-E_{h+1,l}^{\theta}).$$
Moreover, for any $f \in \mathbf{R}^{W_{[A]^{\mf c}}}_{\mathbb{A}}$ and $g \in \mathbf{R}^{W_{[B]^{\mf c}}}_{\mathbb{A}}$,  we have
$$f\bar \star g=fg \in \mathbf{R}^{W_{[A\circ B]^{\mf c}}}_{\mathbb{A}}.$$

\item[(b)] Let $A=E_{h,h-1}^{\theta}(\mathbf{v},a)$, and $l=\min\{i\mid b_{h-1,i}\neq 0\}$.
If $b_{h-1,l} \geq a$, then
$$A\circ B=B+a(E_{hl}^\theta-E_{h-1,l}^{\theta}),$$
Moreover, for any $f \in \mathbf{R}^{W_{[A]^{\mf c}}}_{\mathbb{A}}$ and $g \in \mathbf{R}^{W_{[B]^{\mf c}}}_{\mathbb{A}}$,   we have
$$f\bar \star g=fg \in \mathbf{R}^{W_{[A\circ B]^{\mf c}}}_{\mathbb{A}}.$$
\end{itemize}
\end{prop}
\begin{proof}
We show (a) first. By Bruhat order or  \cite[Lemma 3.9]{BKLW14}, the first part of (a) follows. We now show the second part of (a). Let
$$\mathcal{T} = \{(F,F',F'') \in p_{12}^{-1}(\mathcal{O}_{A})\cap p_{23}^{-1}(\mathcal{O}_{ B}) \mid (F,F'') \in \mathcal{O}_{A\circ B}\}.$$
Consider the projection map $p_{13}: \mathcal{T} \rightarrow \mathcal{O}_{A\circ B}, (F, F', F'') \mapsto (F,F'')$.
    If $b_{hk}=0, \forall\  k \geq l$, the fact that the coefficient of leading term is 1 in  \cite[Proposition 3.3]{BKLW14} implies that $p_{13}$ is an isomorphism.
 The second part of (a) follows from Corollary 3 in \cite{V}.
 Statement (b) can be proved similarly.
\end{proof}

Let $p_1: \mathcal{O}_{E_{i,i+ 1}^{\theta}(\mathbf{v},1)} \rightarrow \mcal{F}_{\mathbf{v} + \mathbf{e}_i  + \mathbf{e}_{N + 1-i} }$
(resp. $q_1: \mathcal{O}_{E_{i+1,i }^{\theta}(\mathbf{v},1)} \rightarrow \mcal{F}_{\mathbf{v} + \mathbf{e}_{i+1}  + \mathbf{e}_{N-i} }$) be the first projection.
 Let $T^{*}p_1$ (resp. $T^{*}q_1$) be the relative cotangent sheaf along the fibers of $p_1$ (resp. $q_1$ )
 and $\mathrm{Det}(T^{*}p_1)$ (resp. $\mathrm{Det}(T^{*}q_1)$) the determinant bundle on  $\mathcal{O}_{E_{i,i+ 1}^{\theta}(\mathbf{v},1)}$ (resp. $ \mathcal{O}_{E_{i+1,i }^{\theta}(\mathbf{v},1)}$).
Let $\pi_1: \mcal{Z}_{E_{i,i+1}^{\theta}(\mathbf{v},1)} \rightarrow \mathcal{O}_{E_{i,i+ 1}^{\theta}(\mathbf{v},1)} $
(resp. $\rho_1:  \mcal{Z}_{E_{i+1,i}^{\theta}(\mathbf{v},1)} \rightarrow  \mathcal{O}_{E_{i+1,i }^{\theta}(\mathbf{v},1)}$) be the first projection.
We define
\begin{equation*}
  \begin{split}
\mathscr{E}_{i,\mathbf{v}} =  \pi_1^{*}(\mathrm{Det}(T^{*}p_1)),\quad
\mathscr{F}_{i,\mathbf{v}} =  \rho_1^{*}(\mathrm{Det}(T^{*}q_1)).
  \end{split}
\end{equation*}
  For $i = 1,2,\ldots, n+1$,   let $\mathscr{H}_{i,\mathbf{v}}$ be the sheaf supported on the $\mcal{Z}_{\text{diag}(\mathbf{v})}$, which corresponds  to $q^{v_i}$ via the isomorphism  $K^{\mathbb{C}^{*}\times G }(\mcal{Z}_{\text{diag}(\mathbf{v})})\cong \mathbf{R}_{\mathbb{A}}^{W_{\mathbf{v}^{\mf c}}}$.
  Set
  \begin{equation*}
    \begin{split}
\mathscr{E}_{i}= \sum\limits_{\mathbf{v}\in \Lambda_{\mf c}} (-q)^{-v_i}\mathscr{E}_{i,\mathbf{v}}, &\quad  \mathscr{F}_{i}= \sum\limits_{\mathbf{v}\in \Lambda_{\mf c} } (-q)^{-v_{i+1}-\delta_{n,i}}\mathscr{F}_{i,\mathbf{v}},\\
 \mathscr{H}_i = &\sum\limits_{\mathbf{v} \in \Lambda_{\mf c}} \mathscr{H}_{i,\mathbf{v}}.
    \end{split}
  \end{equation*}

\begin{theorem}\label{generator C}
The convolution algebra $(K^{\mathbb{C}^{*} \times G}(\mcal Z)_{\mathbb{K}},\star)$ is generated by $\mathscr{E}_i$,\ $\mathscr{F}_i$ and the sheaves supported on the orbits indexed by a diagonal matrix in $\Xi_d$.
\end{theorem}

\begin{proof}
The proof is similarly as that for \cite[Proposition 10]{V98}.
For readers' convenience, we repeat it here.
Fix a matrix $C = (c_{ij}) \in \Xi_d$. We prove it by induction
on $$l(C) = \sum\limits_{i > j}\left(
                                 \begin{array}{c}
                                   |i - j| + 1 \\
                                   2\\
                                 \end{array}
                               \right)
c_{ij}.$$
If $l(C) = 0$ or $l(C) = 1$, then $C$ itself is a generator of $K^{\mathbb{C}^{*} \times G}(Z_C)_{\mathbb{K}}$.
There is nothing to show.
We now assume that $l(C)>1$.
We only need show it for the case that $h > n$.  In this case $E_{h+1,h}^{\theta} = E_{N-h,N+1-h}^{\theta}$. Let
$$(h,l) =\text{min}\{(i,j)\mid 1 \leq j < i \leq n,\ c_{ij} \neq 0\}, $$
where
$$(i,j) \geq (s,t) \Leftrightarrow j > t ~\text{or}~ (j = t ~\text{and}~ i> s).$$
Let
$B = C + c_{hl}(E_{h-1,l}^{\theta} - E_{h,l}^{\theta}) \quad \text{and} \quad  A = E_{h,h-1}^{\theta}(\mathbf{v} - c_{h,l}\mathbf{e}_h - c_{h,l}\mathbf{e}_{2n+2-h},c_{h,l}).$
By Proposition \ref{inducpropo}, we have
$$f \bar \star g = fg \in \mathbf{R}_{\mathbb{A}}^{W_{[C]^{\mf c}}}, \quad \forall \ f \in \mathbf{R}_{\mathbb{A}}^{W_{[A]^{\mf c}}},\ g\in \mathbf{R}_{\mathbb{A}}^{W_{[B]^{\mf c}}}.$$

If $(h-1,l) \neq (n+1,n+1)$, by the same argument  as that for \cite[Proposition 10]{V98}, the map
$\bar \star:  \mathbf{R}_{\mathbb{A}}^{W_{[A]^{\mf c}}} \otimes  \mathbf{R}_{\mathbb{A}}^{W_{[B]^{\mf c}}} \longrightarrow  \mathbf{R}_{\mathbb{A}}^{W_{[C]^{\mf c}}}$ is surjective.

If $(h-1,l) = (n+1,n+1)$,
for any partition $ (v_1,v_2,v_1)$ of $2d$,
the algebra $\mathbf{R}_{\mathbb{A}}^{W_{[2d]^{\mf c}}}$ is generated by $\mathbf{R}^{\mathbb{Z}_2^{d} \ltimes  S_d}$ and $\mathbf{R}^{S_{v_1}}$.
 Thus  the surjectivity of the map
  $\bar \star:  \mathbf{R}_{\mathbb{A}}^{W_{[A]}} \otimes  \mathbf{R}_{\mathbb{A}}^{W_{[B]}} \longrightarrow  \mathbf{R}_{\mathbb{A}}^{W_{[C]}}$
  follows from the surjectivity of the map
$$\mathbf{R}_{\mathbb{A}}^{S_{a_{h,h-1}}} \otimes \mathbf{R}_{\mathbb{A}}^{\mathbb{Z}_2^{|b_{n+1,n+1}^{\mbf c}|} \ltimes  S_{b_{n+1,n+1}^{\mbf c}}} \longrightarrow \mathbf{R}_{\mathbb{A}}^{\mathbb{Z}_2^{|C_{n+1,n+1}^{\mbf c}|} \ltimes  S_{C_{n+1,n+1}^{\mbf c}}} . $$

 Similar to Example \ref{fgenerator}, any sheaf supported on $Z_{E^{\theta}_{i+1,i}(\mathbf{v},a)}$ can be obtaind by the convolution product of the sheaves supported on  $Z_{E^{\theta}_{i+1,i}(\mathbf{v} + k \mathbf{e}_i - k \mathbf{e}_{2n+2-i},1)},$ $ k = 0,1,\cdots, a-1$. Thus  $(K^{\mathbb C^{*} \times G}(Z_C)_{\mathbb{K}},\star)$ can be generated by classes of the sheaves supported on the irreducible components $Z_A$ with $A \in  \Xi_d$ diagonal or type $E_{i+1,i}^{\theta}(\mathbf{v},1)$.

 By Proposition \ref{diagonal action}, for any $\mathbf{v} \in \Lambda_{\mf c}$, the map
$$\bar \star:\mathbf{R}^{W_{[\mathbf{v} + \mathbf{e}_{i+1} + \mathbf{e}_{2n+1-i}]^{\mf c}}}\otimes [\mathscr{F}_{i,\mathbf{v}}]\otimes \mathbf{R}^{W_{[\mathbf{v} + \mathbf{e}_{i+1} + \mathbf{e}_{2n+1-i}]^{\mf c}}}\longrightarrow \mathbf{R}^{W_{[E_{i+1,i}^{\theta}(\mathbf{v},1)]^{\mf c}}}$$
is surjective. The theorem follows.
\end{proof}

\subsection{The $K^{\mathbb{C}^{*} \times G}(\mcal{Z})$-module structure on $K^{\mathbb{C}^{*} \times G}(\mcal{M})$}
\label{Generate relation of type C}

Recall $\mcal Z=\mcal M\times_{\mcal N} \mcal M$. Convolution product naturally defines an action of $K^{\mathbb{C}^{*} \times G}(\mcal{Z})$ on $K^{\mathbb{C}^{*} \times G}(\mcal{M})$ as follows.
 $$ \rho: K^{\mathbb{C}^{*} \times G}(\mcal{Z}) \otimes K^{\mathbb{C}^{*} \times G}(\mcal{M}) \longrightarrow  K^{\mathbb{C}^{*} \times G}(\mcal{M}), \quad f\otimes g \mapsto f\star g.$$
\begin{lemma}~\cite[claim.7.6.7]{CG}
\label{faithful lemma}
The action $\rho$ gives a faithful representation of $K^{\mathbb{C}^{*} \times G}(\mcal{Z})$ on $K^{\mathbb{C}^{*} \times G}(\mcal{M})$.
\end{lemma}

In the rest of this subsection, we shall identify $K^{\mathbb{C}^{*} \times G}(\mcal{M})$ with $\oplus_{\mbf v\in \Lambda_{\mathfrak c}} \mbf R^{W_{[\mbf v]^{\mathfrak c}}}$.

\subsubsection{ Action of $\mathscr{E}_{i,\mathbf{v}}$ on $K^{\mathbb{C}^{*} \times G}(\mcal{M})$}

\begin{prop}\label{e'sa action1}
Fix a  partition $\mathbf{v} \in \Lambda_{\mathfrak c}$.
 For any $ 1 \leq i \leq n$ ,  $f \in  K^{\mathbb{C}^{*}\times G}(T^*\mathcal{F}_{\mathbf{v} + \mathbf{e}_{i+1} + \mathbf{e}_{2n+1-i}})$, we have
 \begin{equation*}
\begin{split}
 &\rho(\mathscr{E}_{i,\mathbf{v}}\otimes f)
  =  W_{[\mathbf{v}+e_{i}+e_{2n+2-i}]^{\mf c}}/W_{[E_{i,i+1}^{\theta}(\mathbf{v},1)]^{\mf c}}(\prod_{\bar{v}_{i-1} < t \leq \bar{v}_{i}} \frac{ (q-q^{-1}x_{t}/x_{\bar{v}_{i}+1})}{(1-x_t/x_{\bar{v}_{i}+1})} \cdot f).
 \end{split}
 \end{equation*}
 \end{prop}
\begin{proof}
Equation \eqref{cotangent formula} implies that $T_{p_{1}}^{*}=\sum_{ \bar{v}_{i-1} < t \leq  \bar{v}_{i}} x_t/x_{\bar{v}_{i}+1}$.
Hence, we further have
 \begin{eqnarray*}
&&\text{Det}(T_{p_{1}}^{*})=\prod_{ \bar{v}_{i-1} < t \leq  \bar{v}_{i}} x_t/x_{\bar{v}_{i}+1}, \\
&&[\bigwedge_{q^2}T_{p_{1}}]= \prod_{\bar{v}_{i-1} < t \leq  \bar{v}_{i}}(1-q^2x_{\bar{v}_{i}+1}/x_t), \\
&&[\bigwedge T^{*}_{p_{1}}]= \prod_{\bar{v}_{i-1} < t \leq  \bar{v}_{i}}(1-x_t/x_{\bar{v}_{i}+1}).
\end{eqnarray*}

By Proposition \ref{pushforwad C} and Corollary 4 in \cite{V98}, the proposition follows.
\end{proof}

\subsubsection{ Action of $\mathscr{F}_{i,\mathbf{v}}$ on $K^{\mathbb{C}^{*} \times G}(\mcal{M})$}

\begin{prop} \label{f's action1}  Let $\mathbf{v}= (v_1, \cdots, v_N)$ be a partition of $2d-2$ such that $v_i = v_{N+1-i}$.

(a)  If $ 1 \leq i < n$ ,  for any $f \in  K^{\mathbb{C}^{*} \times G}(T^*\mathcal{F}_{\mathbf{v} + \mathbf{e}_{i} + \mathbf{e}_{2n+2-i}})$, we have
 \begin{equation*}
\begin{split}
 &\rho(\mathscr{F}_{i,\mathbf{v}}\otimes f)
  =   W_{[\mathbf{v}+e_{i+1}+e_{2n+1-i}]^{\mf c}}/W_{[E_{i+1,i}^{\theta}(\mathbf{v},1)]^{\mf c}}(\prod_{\bar{v}_{i} + 1 < t \leq  \bar{v}_{i+1}+1} \frac{ (q-q^{-1}x_{\bar{v}_{i}+1}/x_{t})}{(1-x_{\bar{v}_{i}+1}/x_{t})} \cdot f).
 \end{split}
 \end{equation*}

 (b) If $i = n$,  for any $f \in  K^{\mathbb{C}^{*} \times G}(T^*\mathcal{F}_{\mathbf{v} + \mathbf{e}_{n} + \mathbf{e}_{n+2}})$, we have
 \begin{equation*}
\begin{split}
 &\rho(\mathscr{F}_{n,\mathbf{v}}\otimes f)\\
  = &  W_{[\mathbf{v}+2e_{n+1}]^{\mf c}}/W_{[E_{n+1,n}^{\theta}(\mathbf{v},1)]^{\mf c}}((\prod_{\bar{v}_{n} + 1 < t \leq  d} \frac{q-q^{-1}x_t^{-1}x_{\bar{v}_{n}+1}}{1-x_t^{-1}x_{\bar{v}_{n}+1}}\frac{q-q^{-1}x_tx_{\bar{v}_{n}+1}}{1-x_tx_{\bar{v}_{n}+1}})\frac{q-q^{-1}x_{\bar{v}_{n}+1}^2}{1-x_{\bar{v}_{n}+1}^2} \cdot f).
 \end{split}
 \end{equation*}
\end{prop}

\begin{proof}
If $1\leq i < n$,  Equation \eqref{cotangent formula} implies that $T_{p_{1}}^{*}=\sum_{ \bar{v}_{i} + 1 < t \leq  \bar{v}_{i+1}+1} x_{\bar{v}_{i}+1}/x_{t}$.
Hence, we have
\begin{align*}
 &\text{Det}(T_{p_{1}}^{*})=\prod_{ \bar{v}_{i} + 1 < t \leq  \bar{v}_{i+1}+1} x_{\bar{v}_{i}+1}/x_{t},\\
 &[\bigwedge_{q^2}T_{p_{1}}]= \prod_{\bar{v}_{i} + 1 < t \leq  \bar{v}_{i+1}+1}(1-q^2 x_{t}/x_{\bar{v}_{i}+1}),\\
 & [\bigwedge T^{*}_{p_{1}}]= \prod_{ \bar{v}_{i} + 1 < t \leq  \bar{v}_{i+1}+1}(1-x_{\bar{v}_{i}+1}/x_{t}).
\end{align*}

In the case that $i = n$, we have
\begin{align*}
& T_{p_{1}}^*= \sum_{\bar{v}_n+1 < t \leq d}( x_t^{-1}x_{\bar{v}_{n}+1} +x_tx_{\bar{v}_{n}+1})+ x_{\bar{v}_{n}+1}^2,\\
&\text{Det}(T_{p_{1}}^*) = \prod_{\bar{v}_n+1 < t \leq d}(x_t^{-1}x_{\bar{v}_{n}+1} \cdot x_tx_{\bar{v}_{n}+1})\cdot x_{\bar{v}_{n}+1}^2,\\
& \bigwedge_{q^2} T_{p_{1}} =(\prod_{\bar{v}_n+1 < t \leq d}(1-q^{2}x_tx^{-1}_{\bar{v}_{i}+1})(1-q^{2}(x^{-1}_tx^{-1}_{\bar{v}_{n}+1}))) (1-q^{2}x^{-2}_{\bar{v}_{n}+1}),\\
& \bigwedge T_{p_{1}}^* =(\prod_{\bar{v}_n+1 < t \leq d}(1-x_t^{-1}x_{\bar{v}_{n}+1})(1-x_tx_{\bar{v}_{n}+1})) (1-x_{\bar{v}_{n}+1}^2).
\end{align*}
By Proposition \ref{pushforwad C} and Corollary 4 in \cite{V98} again, the proposition follows.
\end{proof}

\section{The coideal subalgebra of $\U_q(\mf{gl}_N)$}

\subsection{The algebra $\Ub$}
  Let us recall the  presentation of the algebra $\Ub$ (denoted by $\U^{\jmath}$ therein) from \cite{BKLW14}.
\begin{defn}\label{def3.1}
The algebra $\Ub$ is the unitary associative $\mathbb C(q)$-algebra generated by
$$
e_i, f_i, h_a^{\pm 1},   \quad
\forall i\in [1,n],\  a\in [1, n+1],
$$
subject to  the following relations
\begin{eqnarray*}
&&h_ah_b=h_b h_a,\quad h_ah_a^{-1}=1,\\
&&h_ae_ih_a^{-1}=q^{ \delta_{a, i} -\delta_{a,i+1}  - \delta_{2n+2-a,i+1}}e_i, \\
&&h_af_ih_a^{-1}=q^{- \delta_{ai} + \delta_{a,i+1} + \delta_{2n+2-a,i+1}}f_i,\\
&& e_if_j-f_je_i=\delta_{ij}\frac{h_ih_{i+1}^{-1}-h_i^{-1}h_{i+1}}{q-q^{-1}},  \quad \quad \,   {\rm if} \ i, j\neq n,\\
&&e_ie_j=e_je_i,\quad f_if_j=f_jf_i,\quad \hspace{2.1cm}  {\rm if}\ |i-j|>1,\\
&&e_i^2e_j-(q+q^{-1})e_ie_je_i+e_je_i^2=0,\quad \hspace{.51cm}  {\rm if}\ |i-j|=1,\\
&&f_i^2f_j-(q+q^{-1})f_if_jf_i+f_jf_i^2=0,\quad \hspace{.71cm}  {\rm if}\ |i-j|=1,\\
&&e^2_nf_n+f_ne_n^2=(q+q^{-1})(e_nf_ne_n-e_n(qh_nh_{n+1}^{-1}+q^{-1}h_n^{-1}h_{n+1})), \\
&& f_n^2e_n+e_nf_n^2=(q+q^{-1})(f_ne_nf_n-(qh_nh_{n+1}^{-1}+q^{-1}h_n^{-1}h_{n+1})f_n)
\end{eqnarray*}
 for $i,j = 1,2,\cdots,n$ and  $a,b = 1,2,\cdots,n+1$.
\end{defn}
 The algebra $\Ub$ is a coideal subalgebra of $\U_q(\mrk{gl}_N)$ with coideal structure $\Delta^{\mf b} :\Ub \rightarrow \Ub \otimes \U_q(\mrk{gl}_N) $ given by
 \begin{align*}
\begin{split}
\Delta^{\mf b}( e_i)
&= e_i \otimes  H_{i+1}  H^{-1}_{2n+1-i} +  h^{-1}_{i+1} \otimes  E_iH^{-1}_{2n+1-i} +  h_{i+1} \otimes  F_{2n+1-i}  H_{i+1}.  \\
 \Delta^{\mf b} (f_i)
 & = f_i \otimes H^{-1}_{i} H_{2n+2-i} +  h_i\otimes  F_i H_{2n+2-i} +  h^{-1}_{i} \otimes  E_{2n+1-i}  H^{-1}_{i}. \\
 \Delta^{\mf b} ( h_i) & =  h_i \otimes  H_i H^{-1}_{2n+1-i}.
\end{split}
\end{align*}

\subsection{A polynomial representation of $\Ub$}
 Recall  $\mathbf{R} = \mathbb{C}[x_1^{\pm 1},x_2^{\pm 1},\cdots,x_d^{\pm 1}]$. Set $\mathbf{K} = \oplus _{\mathbf{v} \in \Lambda_{\mf c}} \mathbf{R}^{W_{[\mathbf{v}]^{\mf c}}}$. Let
$$\hat{h}_i \in \oplus_{\mathbf{v} \in \Lambda_{\mf c}}\mathrm{Hom} (\mathbf{R}_{\mathbb{A}}^{W_{[\mathbf{v}]^{\mf c}}},\mathbf{R}_{\mathbb{A}}^{W_{[\mathbf{v}]^{\mf c}}}),$$
 $$\hat{e}_i,~\hat{f}_i \in \oplus_{\mathbf{v} \in \Lambda_{\mf c}}\mathrm{Hom} (\mathbf{R}_{\mathbb{A}}^{W_{[\mathbf{v}]^{\mf c}}},\mathbf{R}_{\mathbb{A}}^{W_{[\mathbf{v}\pm e_i \pm e_{2n+2-i} \mp e_{i+1} \mp e_{2n+1-i}]^{\mf c}}})$$
 be the following operators
 $$\hat{h}_i(f) = q^{v_i}f,$$
  \begin{equation*}
\begin{split}
 &\hat{e}_i(f)
  =  W_{[\mathbf{v}']^{\mf c}}/(W_{[\mathbf{v}']^{\mf c}}\cap W_{[\mathbf{v}]^{\mf c}})(\prod_{\bar{v}_{i-1} < t \leq \bar{v}_{i}} \frac{ (q-q^{-1}x_{t}/x_{\bar{v}_{i}+1})}{(1-x_t/x_{\bar{v}_{i}+1})} \cdot f),
 \end{split}
 \end{equation*}

 $ \hat{f}_i(f) =\left\{\begin{array}{ll}
 W_{[\mathbf{v}'']^{\mf c}}/(W_{[\mathbf{v}'']^{\mf c}}\cap W_{[\mathbf{v}]^{\mf c}})(\prod_{\bar{v}_{i}  < t \leq  \bar{v}_{i+1}} \frac{ (q-q^{-1}x_{\bar{v}_{i}}/x_{t})}{(1-x_{\bar{v}_{i}}/x_{t})} \cdot f)& \text{if} \ i < n, \\[.15in]
W_{[\mathbf{v}'']^{\mf c}}/(W_{[\mathbf{v}'']^{\mf c}}\cap W_{[\mathbf{v}]^{\mf c}})((\prod_{\bar{v}_{n}  < t \leq  d} \frac{q-q^{-1}x_t^{-1}x_{\bar{v}_{n}+1}}{1-x_t^{-1}x_{\bar{v}_{n}}}\frac{q-q^{-1}x_tx_{\bar{v}_{n}}}{1-x_tx_{\bar{v}_{n}}})\frac{q-q^{-1}x_{\bar{v}_{n}}^2}{1-x_{\bar{v}_{n}}^2} \cdot f)& \text{if}\  i=n,
 \end{array}   \right.
$
where $f \in \mathbf{R}_{\mathbb{A}}^{W_{[\mathbf{v}]^{\mf c}}}$, $\mathbf{v}'=\mathbf{v} +  e_i + e_{2n+2-i} - e_{i+1} - e_{2n+1-i}$, $\mathbf{v}'' = \mathbf{v}- e_i - e_{2n+2-i}  +  e_{i+1}  +  e_{2n+1-i}$.
\begin{prop}\label{polynomial rep C}
The map
$$e_i \mapsto \hat{e}_i,~f_i \mapsto \hat{f}_i,~h_i \mapsto \hat{h}_i, $$
 defines a representation of $\Ub$ on $\mathbf{K}_{\mathbb{K}}$.
\end{prop}
\begin{proof}
We shall give a detail proof in appendix.
\end{proof}

\subsection{K-theoretic approach to $\Ub$}

By Propositions \ref{e'sa action1},  \ref{f's action1} and  \ref{polynomial rep C}, we have the following proposition.
\begin{prop}
\label{representation prop}
There is an algebra homomorphism $\Ub \longrightarrow End(K^{\mathbb{C}^{*} \times G}(\mcal M)_{\mathbb{K}})$ sending
$$e_i \mapsto \mathscr{E}_{i},\  f_i \mapsto \mathscr{F}_i,\ h_i^{\pm} \mapsto \mathscr{H}_i^{\pm}.$$
\end{prop}

By Lemma \ref{faithful lemma} and  Proposition \ref{representation prop}, we obtain the following main result of this section.
\begin{theorem}
\label{main theorem C}
The assignment
$$e_i \mapsto \mathscr{E}_{i}, \ f_i \mapsto \mathscr{F}_i,\ h_i^{\pm} \mapsto \mathscr{H}_i^{\pm}$$
can be extended to an algebra homomorphism $\Ub \rightarrow K^{\mathbb{C}^{*} \times G}(\mcal{Z})_{\mathbb{K}}$.
\end{theorem}

\subsection{The coideal structure}
\label{coideal C}

For any semisimple element $ \alpha =(c,s) \in \mathbb{C}^{*}\times G$,
let $\epsilon: \mathbf{R}(\mathbb{C}^{*}\times G) \rightarrow \mathbb{C},\ f \mapsto f(\alpha)$ be the evaluation map
 and $\mathbb{C}_\alpha$ be the 1-dimensional $\mathbf{R}(\mathbb{C}^{*}\times G)$-module with $f$ acting via multiplication by $f(\alpha)$.
 For any algebraic variety $Y$,
denote by $Y^{\alpha}$  the fixed point subvariety of $Y$ by $\alpha$.
Assume that $s $ is a semisimple element in $G$ with $2k - 1$ distinct eigenvalues $\{\lambda_{1}^{\pm 1},\lambda_2^{\pm 1}, \cdots, \lambda_{k-1}^{\pm 1}, 1\}$ with multiplicity $d_1, \cdots,d_{k-1}, 2d_k$, respectively and set $\alpha=(1,s)\in \mathbb{C}^{*}\times G$. There are natural isomorphisms
 \begin{align*}
&G^s\cong GL_{d_1}\times GL_{d_2}\times \dots\times GL_{d_{k-1}}\times Sp_{d_k},\\
&\mcal{Z}^\alpha \cong \mcal{Z}_{d_1}^{\mf a}\times \mcal{Z}_{d_2}^{\mf a} \times \dots \times \mcal{Z}_{d_{k-1}}^{\mf a} \times \mcal{Z}_{d_{k}}^{\mf c},
 \end{align*}
where 
$\mcal{Z}_{d_i}^{\mf a}$ is the Steinberg variety of type A of rank $d_i$ and $\mcal{Z}_{d_k}^{\mf c}$ is the generalized Steinberg variety of type C of rank $d_k$.
Let $\mathcal{N}_\alpha^{*}$ be the conormal bundle of  the subvariety $\mcal{Z}^\alpha \hookrightarrow \mcal{Z}$. The $\mathbb{C}^{*} \times G$ action on $\mcal{Z}$ induces a natural $\alpha$-action on $\mathcal{N}_\alpha^{*}$. Set
$$\lambda(\mcal{Z}^\alpha) = \sum_k (-1)^k {\rm tr}(\alpha; \wedge^k\mathcal{N}_\alpha^{*}).$$
Let $\mcal{M}_1, ~ \mcal{M}_2$ be the $\mathbb{C}^{*}\times G$ variety $\mcal M = T^*\mathcal F $ and $j:\mcal{Z} \hookrightarrow \mcal{M}_1\times \mcal{M}_2$ be the imbedding.
Define  a morphism $\mathbf{r}_\alpha: \mathbb{C}_\alpha \otimes_{\mathbf{R}(\mathbb{C}^{*}\times G)} K^{\mathbb{C}^{*}\times G}(\mcal{Z}) \rightarrow \mathbb{C}_\alpha \otimes_{\mathbf{R}((\mathbb{C}^{*}\times G)^\alpha)} K^{(\mathbb{C}^{*}\times G)^{\alpha}}(\mcal{Z}^\alpha)$ by
$$c\otimes \mrk{F} \mapsto c\otimes \sum(-1)^k \text{Tor}_{\mathcal{O}_{\mcal{M}_1^\alpha \times \mcal{M}_2^\alpha}}^k(\lambda(\mcal{M}_1^\alpha)^{-1}\times \mathcal{O}_{\mcal{M}_2^\alpha},j_{!}\mrk F).$$
The K\"unneth  formula implies that  there is a natural isomorphism
$$\tau : K^{(\mathbb{C}^{*}\times G)^{\alpha}}(\mcal{Z})\simeq  K^{\mathbb{C}^{*} \times GL_{d_1} } (\mcal{Z}_{d_1}^{\mf a})\otimes \dots\otimes K^{\mathbb{C}^{*} \times GL_{d_{k-1}}  }(\mcal{Z}_{d_{k-1}}^{\mf a}) \otimes  K^{\mathbb{C}^{*}\times G_{d_k} }(\mcal{Z}_{d_{k}}).$$
Then by  the bivariant fixed-point theorem \cite[Theorem 5.7]{GV93} and the K\"unneth formula, we have the following theorem.
\begin{theorem}\label{thm-coideal}
Let $\Delta^{k-1}$ be the composition of $\mathbf{r}_\alpha$ and $\tau$, then
$$\Delta^{k-1}:\mathbb{C}_\alpha \otimes K^{ \mathbb{C}^{*}\times G}(\mcal{Z}) \rightarrow \mathbb{C}_\alpha \otimes K^{\mathbb{C}^{*} \times GL_{d_1}  } (\mcal{Z}_{d_1}^{\mf a})\otimes \dots\otimes K^{\mathbb{C}^{*} \times GL_{d_{k-1}} }(\mcal{Z}_{d_{k-1}}^{\mf a}) \otimes  K^{\mathbb{C}^{*} \times G_{d_k}}(\mcal{Z}_{d_{k}})$$
 is an algebra homomorphism.
\end{theorem}
By Theorem 8 in \cite{V}, there is an algebra homomorphism $\U_q(\mrk{gl}_{d_i}) \rightarrow K^{\mathbb{C}^{*} \times GL_{d_i}  } (\mcal{Z}_{d_i}^{\mf a})$.
Hence the above theorem gives a coideal structure of $\Ub$.

\section{The coideal subalgebra $\Uc$}

\subsection{The algebra $\Uc$}

In \cite[Section 6]{BKLW14}, authors consider another $\imath$-quantum group, $\Uc$ (denoted by ${}^C\U$ therein).
By using perverse sheaves on $\mcal F^{\mathfrak c} \times \mcal F^{\mathfrak c}$, a geometric approach to $\Uc$ is obtained.
It has been shown that $\Uc \simeq \Ub$ in \cite{BKLW14}.
In this section, we shall provide a K-theoretic approach to $\Uc$ by using generalized Steinberg variety of type B.
Although, $\Uc\simeq \Ub$, in general, $K^{\mbb C^*\times Sp(\mbb C)}(\mcal Z^{\mathfrak c})$ is not isomorphic to $K^{\mbb C^*\times O(\mbb C)}(\mcal Z^{\mathfrak b})$.
Moreover, K-theoretic approach to the coideal structure of $\Uc$ is obtained.

\subsection{K-theoretic approach to $\Uc$}
Let $V =\mbb C^{2d + 1}$ with a nondegenerate symmetric bilinear form $( , )$ and $G_{\mf b} = O(V)$.
 Set
$$\Lambda_{\mf b}=\{ \mathbf{v}=(v_i) \in \mbb N^{N} \mid   v_i = v_{N+1-i},\quad \textstyle \sum_{i=1}^{N} v_i = 2d+1 \}.$$
By abuse notations,  for any $W\subseteq V$, we still denote by $W^{\perp}$ the orthogonal complement of $W$ with respect to this symmetric bilinear form.
For any $\mbf v \in \Lambda_{\mf b}$, we further set
$$\mathcal{F}^{\mf b}_{\mathbf{v}}=\{ F=( 0=V_0 \subset V_1  \cdots \subset V_{2n+1}=V)\mid  V_i=V_{2n+1-i}^{\perp} ,\text{dim}( V_{i}/V_{i-1})= v_i, \forall i \}. $$
The componentwise action of $G_{\mf b}$ on $\mcal F^{\mf b}_{\mbf v}$ is transitive.
For any $F \in \mathcal{F}^{\mf b}_{\mathbf{v}}$, let $P_{\mathbf{v}} $ be the stabilizer of $F$ in $G_{\mf b}$.
Therefore, we have $G_{\mf b}/P_{\mathbf{v}}\simeq \mcal F_{\mbf v}$.
Let $\mathcal{F}^{\mf b} = \sqcup _{\mathbf{v} \in \Lambda_{\mf b}} \mathcal{F}^{\mf b}_{\mathbf{v}}$.
Group $G_{\mf b}$ naturally acts on $\mcal F^{\mf b}$.

Similar as Section \ref{steinberg variety of type c}, we define $\mcal M^{\mf b} = T^* \mcal F^{\mf b}$ and $\mcal Z^{\mf b} = \mcal M^{\mf b} \times_{\mcal N} \mcal M^{\mf b}$.
The later is the generalized Steinberg variety of type B.
In this subsection, we shall calculate the convolution algebra
$(K^{\mathbb{C}^{*} \times G_{\mf b}}(\mcal{Z}^{\mf b}), \star)$ similarly as Section 2.

Let $\Pi_d=(a_{ij})_{1\leq i,j \leq N}$ be the set of  all $N\times N$ matrices with  nonnegative integer  entries such that $a_{ij}=a_{N+1-i,N+1-j}$ for $ i,j \leq n+1$ and $\sum_{i,j}a_{ij}=2d+1$.
 Similarly, there is a bijection between  $G_{\mf b}$-obits in $\mathcal{F}^{\mf b} \times \mathcal{F}^{\mf b}$ and $\Pi_d$.
 Moreover, we have the following two bijections
\begin{equation*}
  \begin{split}
\varphi: \Lambda_{\mf b} \rightarrow \Lambda_{\mf c},\quad & \mathbf{v} \mapsto \mbf v-(\delta_{i,n+1})_{i=1,\cdots, N}, \\
 \psi: \Pi_d \rightarrow \Xi_d, \quad & A \mapsto A-E_{n+1,n+1}.
  \end{split}
\end{equation*}
Then there is an order $``\preceq"$ on $\Pi_d(\mathbf{v},\mathbf{w})$, $\forall \mathbf{v},\mathbf{w} \in \Lambda_{\mf b}$ via the above bijections.
Set
$$W_{[\mathbf{v}]^{\mf b}}= W_{[\varphi({\mathbf{v}})]^{\mf c}},\quad W_{[A]^{\mf b}}= W_{[\psi(A)]^{\mf c}}.$$
By Proposition 6.7  in \cite{BKLW14}, the orbit structure of type B is the same as that of type C under the above bijections.
So we can obtain similar results as Section \ref{Generators of type C}.
Define $\mathcal{E}_{i,\mathbf{v}}$ and $\mathcal{F}_{i,\mathbf{v}}$ in the same way as $\mathscr{E}_{i,\mathbf{v}}$ and $\mathscr{F}_{i,\mathbf{v}}$, respectively.
 Let $\mathcal{H}_{i,\mathbf{v}}^{\pm 1}$ be the sheaves supported on the $\mcal{Z}^{\mf b}_{\text{diag}(\mathbf{v})}$, which corresponds to $q^{\pm(v_i - \delta_{n+1,i})}$.
 We set
\begin{equation*}
  \begin{split}
\mathcal{E}_{i}= \sum\limits_{\mathbf{v} \in \Lambda_{\mf b}} (-q)^{-v_i}\mathcal{E}_{i,\mathbf{v}}, &\quad \mathcal{F}_{i}= \sum\limits_{\mathbf{v}\in \Lambda_{\mf b} } (-q)^{-v_{i+1}}\mathcal{F}_{i,\mathbf{v}},\\
\mathcal{H}_i =& \sum\limits_{\mathbf{v} \in \Lambda_{\mf b}} \mathcal{H}_{i,\mathbf{v}}.
  \end{split}
\end{equation*}
 The following is an analogue of  Theorem \ref{generator C}.
  We shall skip the proof..

\begin{theorem}\label{generator B}
The convolution algebra $(K^{\mathbb{C}^{*} \times G_{\mf b}}(\mcal{Z}^{\mf b})_{\mathbb{K}},\star)$ is generated by $\mathcal{E}_i$,\ $\mathcal{F}_i$ and the sheaves supported on the orbits indexed by a diagonal matrix in $\Pi_d$.
\end{theorem}

In order to calculate the relations of $K^{\mathbb{C}^{*} \times G_{\mf b}}(\mcal{Z}^{\mf b})$ in term of the generators in Theorem \ref{generator B},
we still consider the faithful action $\rho_{\mf b}$ of $K^{\mathbb{C}^{*} \times G_{\mf b}}(\mcal{Z}^{\mf b})$ on  $K^{\mathbb{C}^{*} \times G_{\mf b}}(\mcal{M}^{\mf b})$ given by the convolution product
      $$ \rho_{\mf b}: K^{\mathbb{C}^{*} \times G_{\mf b}}(\mcal{Z}^{\mf b}) \otimes K^{\mathbb{C}^{*} \times G_{\mf b}}(\mcal{M}^{\mf b}) \longrightarrow  K^{\mathbb{C}^{*} \times G_{\mf b}}(\mcal{M}^{\mf b}), \quad f\otimes g \mapsto f\star g.$$

For a  partition $\mathbf{v} \in \Lambda_{\mf b}$, the action of $\mathcal{E}_{i,\mathbf{v}}$, $ 1\leq i  \leq n$ (resp. $\mathcal{F}_{j,\mathbf{v}}$, $ 1\leq j  < n$) are given by the same as that of $\mathscr{E}_{i,\varphi(\mathbf{v})}$ (resp. $\mathscr{F}_{j,\varphi(\mathbf{v})}$). Here we only need to give the action of $\mathcal{F}_{n,\mathbf{v}}$ on $K^{\mathbb{C}^{*} \times G_{\mf b}}(\mcal{M}^{\mf b})$.
In this case, we have
 \begin{align*}
&T_{p_{1}}^*= \sum_{\bar{v}_n+1 < t \leq d} x_t^{-1}x_{\bar{v}_{n}+1} +x_tx_{\bar{v}_{n}+1}+ x_{\bar{v}_{n}+1},\\
&\text{Det}(T_{p_{1}}^*) = \prod_{\bar{v}_n+1 < t \leq d}x_t^{-1}x_{\bar{v}_{n}+1} \cdot x_tx_{\bar{v}_{n}+1}\cdot x_{\bar{v}_{n}+1},\\
& \bigwedge T_{p_{1}}^* =(\prod_{\bar{v}_n+1 < t \leq d}(1-x_t^{-1}x_{\bar{v}_{n}+1})(1-x_tx_{\bar{v}_{n}+1})) (1-x_{\bar{v}_{n}+1}),\\
& \bigwedge_q T_{p_{1}}^* =(\prod_{\bar{v}_n+1 < t \leq d}(1-q^{2}x_t^{-1}x_{\bar{v}_{i}+1})(1-q^{2}x_tx_{\bar{v}_{n}+1})) (1-q^{2}x_{\bar{v}_{n}+1}).
 \end{align*}
 For any $f \in  K^{\mathbb{C}^{*} \times G_{\mf b} }(T^*\mathcal{F}^{\mf b}_{\mathbf{v} + \mathbf{e}_{n} + \mathbf{e}_{n+2}})$, we have
 \begin{equation*}
\begin{split}
 &\rho_{\mf b}(\mathscr{F}_{n,\mathbf{v}}\otimes f)\\
  = &  W_{[\mathbf{v}+2e_{n+1}]}/W_{[E_{n+1,n}^{\theta}(\mathbf{v},1)]}((\prod_{\bar{v}_{n} + 1 < t \leq  d} \frac{q-q^{-1}x_t^{-1}x_{\bar{v}_{n}+1}}{1-x_t^{-1}x_{\bar{v}_{n}+1}}\frac{q-q^{-1}x_tx_{\bar{v}_{n}+1}}{1-x_tx_{\bar{v}_{n}+1}})\frac{q-q^{-1}x_{\bar{v}_{n}+1}}{1-x_{\bar{v}_{n}+1}} \cdot f)
 \end{split}
 \end{equation*}

\begin{prop}
\label{representation prop B}
The map
$${e}_i \mapsto \mathcal{E}_{i},\ {f}_i \mapsto \mathcal{F}_i,\ {h}_i^{\pm 1} \mapsto \mathcal{H}_i^{\pm 1}$$
 extends uniquely to a representation of  $\Uc$ on  $K^{\mathbb{C}^{*} \times G_{\mf b}}(\mcal{M}^{\mf b})_{\mathbb{K}}$.
\end{prop}
\begin{proof}
We only need to prove the last two relations in Definition \ref{def3.1} for the case $n = 1$.
Assume $n = 1$, fix  a partition $I = (0,i,2d+1-2i,i)$, and  $Q \in (R(T)\otimes \mathbb{C})^{W_{[I]^{\mf b}}}.$
The proof of proposition is almost the same as that of the Proposition \ref{polynomial rep C} except the following difference.
$$\omega(\pm l, \pm k) = \frac{qx_{ l}^{\pm 1} - q^{-1}x_{k}^{\pm 1}}{x_l^{\pm 1} - x_k^{\pm 1}}, \omega(\pm l, \mp k) = \frac{qx_{ l}^{\pm 1} - q^{-1}x_{k}^{\mp 1}}{x_l^{\pm 1} - x_k^{\mp 1}}, \forall l,k \in [1,d],l \neq k$$
$$ \omega(\pm k, \mp k) = \frac{q - q^{-1}x_{k}^{\mp 1}}{1 - x_k^{\mp 1}}, \forall k \in [1,d].$$
Then we have
 \begin{align*}
 {h}_1{h}_2^{-1}(Q\otimes e_i) = &q^{3i - 2d}Q \otimes e_i, \quad {h}_1^{-1}{h}_2(Q\otimes e_i) =  q^{2d-3i}Q \otimes e_i,\\
  {e}_1(Q\otimes e_i) =&\sum\limits_{l = 1}^{1+i}\prod\limits_{s = 1,s\neq l}^i \omega(l,s)(l,i+1)Q \otimes e_{i+1},\\
  {f}_1(Q\otimes e_i)
 =&\sum\limits_{k = i}^d\prod\limits_{t = i,t\neq k}^d\omega(t,k)\omega(-t,k)\omega(-k,k)(k,i)Q \otimes e_{i-1}\\
 &  +\sum\limits_{k = i}^d\prod\limits_{t = i,t\neq k}^d\omega(t,-k)\omega(-t,-k)\omega(k,-k)[r]_k(k,i)Q \otimes e_{i-1} .
 \end{align*}
 Then the remaining calculation is the same as that of Proposition \ref{polynomial rep C}. We shall skip it.
\end{proof}
By the faithful representation and the Proposition \ref{representation prop B}, we obtain the main following  theorem.
\begin{theorem}\label{main theorem B}
The assignment
$${e}_i \mapsto \mathcal{E}_{i},\  {f}_i \mapsto \mathcal{F}_i,\ {h}_i^{\pm 1} \mapsto \mathcal{H}_i^{\pm 1}$$
can be extended to an algebra homomorphism $\Uc \rightarrow K^{\mathbb{C}^{*} \times G_{\mf b}}(\mcal{Z}^{\mf b})_{\mathbb{K}}$.
\end{theorem}

\subsection{Coideal structure of $\Uc$}
\label{coideal B}

Fix a semisimple matrix $s \in G_{\mf b}$ with $2k - 1$ distinct eigenvalues  $\{\lambda_{1}^{\pm 1},\lambda_2^{\pm 1}, \cdots, \lambda_{k-1}^{\pm 1}, 1\}$ with multiple $d_1, d_2, \cdots,~d_{k-1}, ~2d_k+1$ and set $\alpha=(1,s)\in \mathbb{C}^{*}\times G_{\mf b}$.
There are natural isomorphisms
\begin{align*}
&G_{\mf b}^s\cong GL_{d_1}\times GL_{d_2}\times \dots\times GL_{d_{k-1}}\times G_{{\mf b},d_k},\\
&\mcal{Z}^{{\mf b},\alpha} \cong \mcal{Z}_{d_1}^{\mf a}\times \mcal{Z}_{d_2}^{\mf a} \times \dots \times \mcal{Z}_{d_{k-1}}^{\mf a} \times \mcal{Z}^{\mf b}_{d_{k}}.
\end{align*}
Similarly as Section \ref{coideal C}, we define  a morphism $\mathbf{r}_\alpha: \mathbb{C}_\alpha \otimes_{\mathbf{R}(\mathbb{C}^{*}\times G_{\mf b})} K^{\mathbb{C}^{*}\times G_{\mf b}}(\mcal{Z}_{\mf b}) \rightarrow \mathbb{C}_\alpha \otimes_{\mathbf{R}((\mathbb{C}^{*}\times G_{\mf b})^\alpha)} K^{(\mathbb{C}^{*}\times G_{\mf b})^{\alpha}}(\mcal{Z}_{\mf b}^\alpha)$ by
$$c\otimes \mathfrak{F} \mapsto c\otimes \sum(-1)^k \text{Tor}_{\mathcal{O}_{\mcal{M}_{{\mf b},1}^\alpha \times \mcal{M}_{{\mf b},2}^\alpha}}^k(\lambda(\mcal{M}_{{\mf b},1}^\alpha)^{-1}\times \mathcal{O}_{\mcal{M}_{{\mf b},2}^\alpha},j_{!}\mathfrak{F})$$
Similarly as Theorem \ref{thm-coideal}, we have the following theorem.
\begin{theorem}
\label{theorem-coproduct}
There is an  algebra homomorphism
$$\Delta_{\mf b}^{k-1}:\mathbb{C}_\alpha \otimes K^{\mathbb{C}^{*} \times G_{\mf b}}(\mcal{Z}_{{\mf b}}) \rightarrow \mathbb{C}_\alpha \otimes K^{\mathbb{C}^{*}\times GL_{d_1} } (\mcal{Z}_{d_1}^{\mf a})\otimes \dots\otimes K^{\mathbb{C}^{*} \times GL_{d_{k-1}} }(\mcal{Z}_{d_{k-1}}^{\mf a}) \otimes  K^{ \mathbb{C}^{*}\times G_{{\mf b},d_k}}(\mcal{Z}^{\mf b}_{d_{k}}).$$
\end{theorem}

\subsection{Li's conjecture}

In \cite{L18}, Li defines $\sigma$-quiver variety for the Satake diagram without black dot,
which is a fixed point subvariety  of Nakajima quiver variety corresponding to underlying Dynkin diagram $\Gamma$.
 We refer Section 4 in \cite{L18} for detail definitions.
 Let $U(\mathfrak g_{\Gamma}^{\sigma})$ be the envelope algebra of the $\sigma$-fixed subalgebra of Lie algebra $\mathfrak g_{\Gamma}$ associated to $\Gamma$, $\mathfrak Y_{\zeta}(\mathbf w)$ be the Steinberg type $\sigma$-quiver variety,  and $H_{top}(\mathfrak Y_{\zeta}(\mathbf w))$ be the top Borel-Moore homology. Li made the following conjecture.

\begin{conj}~\cite[Conjecture 5.3.4]{L18}
\label{conj}
There is a nontrivial algebra homomorphism $\U(\mathfrak g_{\Gamma}^{\sigma})\rightarrow H_{top}( \mathfrak Y_{\zeta}(\mathbf w))$.
\end{conj}

\begin{theorem}
   In the case with the Satake diagram in (1), Conjecture \ref{conj} holds.
\end{theorem}

\begin{proof}
 It is proved that cotangent bundle of the isotropic flag variety of type B/C is isomorphic to certain special $\sigma$-quiver variety    \cite[Theorem 6.2.1]{L18}.
 From \cite[Lemma 4.2.10]{L18},  we can know that $\mathfrak S_{0}(\mathbf{w})$ is the nilpotent variety of  $\mathfrak{o}_{2d+1}$ or $\mathfrak{sp}_{2d}$.
 Thus the Steinberg type $\sigma$-quiver variety $\mathfrak Y_{\zeta}(\mathbf w)$ is isomorphic to the generalized Steinberg variety of type B/C.
   Hence,  $\U(\mathfrak g^{\sigma}_{\Gamma}) = \Ub_{q=1} (resp. \Uc_{q = 1})$, where $\Ub_{q = 1}$ (resp. $\Uc_{q = 1}$) is  the $\mathbb{C}^{*}$-algebras obtained by specializing  $\Ub$ (resp. $\Uc$) to $q = 1$.

Fix  semisimple elements $\alpha = (t,s) \in \mathbb{C}^*\times G$ and $ \beta  =(t_1,s_1) \in \mathbb{C}^*\times G_{\mf b}$ , the bivariant localization theorem \cite[Theorem 5.11.10]{CG} and the bivariant Riemann-Roch theorem \cite[Theorem 5.11.11]{CG} imply that there are surjective algebra homomorphisms
$$K^{\mathbb{C}^* \times G}(\mcal Z^{\mf c}) \twoheadrightarrow H_{*}(\mcal{Z}^{\mf c,\alpha}), $$
$$K^{\mathbb{C}^* \times G_{\mf b}}(Z^{\mf b}) \twoheadrightarrow H_{*}(\mcal{Z}^{\mf b,\beta}),$$
where $H_{*}(\mcal{Z}^{\mf c, \alpha})$ and $H_{*}(\mcal{Z}^{\mf b, \beta})$  are the Borel-Moore homology.
For $\alpha = \beta = (1,1)$, by   Theorem \ref{main theorem C} and  Theorem \ref{main theorem B} , we have two algebra homomorphisms
$$\Ub_{q=1} \rightarrow K^{\mathbb{C}^* \times G}(\mcal{Z}^{\mf c})_{q=1} \twoheadrightarrow H_{*}(\mcal{Z}^{\mf c, \alpha}),$$
$$\Uc_{q = 1} \rightarrow K^{\mathbb{C}^* \times G_{\mf b}}(\mcal{Z}^{\mf b})_{q=1} \twoheadrightarrow H_{*}(\mcal{Z}^{\mf b, \alpha}).$$
Thus the theorem follows.
\end{proof}

\appendix
\allowdisplaybreaks
\section{The proof of Proposition \ref{polynomial rep C}}

We only need to prove the last two relations in the definition of $\U^{\mf b}$, the others are straightforward.  it is suffice to prove it for the case $n = 1$.
Fix  a partition $I = (0,i,2d-2i,i)$, and  $Q \in (R(T)\otimes \mathbb{C})^{W_{[I]}}.$ Set
$$\omega(\pm l, \pm k) = \frac{qx_{ l}^{\pm 1} - q^{-1}x_{k}^{\pm 1}}{x_l^{\pm 1} - x_k^{\pm 1}}, \omega(\pm l, \mp k) = \frac{qx_{ l}^{\pm 1} - q^{-1}x_{k}^{\mp 1}}{x_l^{\pm 1} - x_k^{\mp 1}},\  \forall l,k \in [1,d].$$
Then we have
\begin{align*}
k_1(Q\otimes e_i) =&h_1h_2^{-1}(Q\otimes e_i)= q^{3i - 2d}Q \otimes e_i,\\
k_1^{-1}(Q\otimes e_i) = &h_1^{-1}h_2(Q\otimes e_i)= q^{2d-3i}Q \otimes e_i,\\
 e_1(Q\otimes e_i)
 = &\sum\limits_{l = 1}^{1+i} (l,1+i)(\prod\limits_{s = 1}^i\frac{qx_{i+1} - q^{-1}x_s}{x_{1+i} - x_s}Q)\otimes e_{i+1}\\
 =&\sum\limits_{l = 1}^{1+i}\prod\limits_{s = 1,s\neq l}^i \omega(l,s)(l,i+1)Q \otimes e_{i+1},\\
 f_1(Q\otimes e_i)
 =& \sum\limits_{k = i}^d \mathbb{Z}_{2,k}\times (k,i)[(\prod\limits_{t = i+1}^d\frac{qx_t - q^{-1}x_i}{x_t - x_i}\frac{qx_t^{-1} - q^{-1}x_i}{x_t^{-1} - x_i})\frac{qx_i^{-1} - q^{-1}x_i}{x_i^{-1} - x_i}Q]\otimes e_{i-1}\\
 =&\sum\limits_{k = i}^d\prod\limits_{t = i,t\neq k}^d\omega(t,k)\omega(-t,k)\omega(-k,k)(k,i)Q \otimes e_{i-1}\\
 & \quad +\sum\limits_{k = i}^d\prod\limits_{t = i,t\neq k}^d\omega(t,-k)\omega(-t,-k)\omega(k,-k)[1]_k(k,i)Q \otimes e_{i-1} .
 \end{align*}
We now show the penultimate identity in Definition \ref{def3.1}. By a direct calculation, we have
 \begin{align*}
 &e_1^{2}f_{1}(Q\otimes e_i)\\
 = &\sum_{l = 1}^{i}\sum_{k = i+1}^{d} \prod_{n=1}^{i}\prod_{s=1,s\neq l}^i \prod _{t=i+1,t\neq k}^d \omega (i+1,n)\omega (l,s)\omega (t,k) \omega (-t,k) \omega(-k,k)\\
 &\quad\quad\quad \quad\quad \quad \quad\quad\quad\quad\quad \quad \omega(l,k)\omega(-l,k)(l,k)Q\otimes e_{i+1}\\
 &+\sum_{l = 1}^{i}\prod_{n=1}^{i}\prod_{s=1,s\neq l}^i \prod _{t=i+1}^d \omega (i+1,n)\omega (l,s)\omega (t,l) \omega (-t,l) \omega(-l,l)Q \otimes e_{i+1}\\
 &+ \sum_{l = 1}^{i}\sum_{k = i+1}^{d} \prod_{n=1}^{i}\prod_{s=1,s\neq l}^i \prod _{t=i+1,t\neq k}^d \omega (i+1,n)\omega (l,s)\omega (t,-k) \omega (-t,-k) \omega(k,-k)\\
 &\quad\quad\quad \quad\quad \quad \quad\quad\quad\quad\quad \quad \quad \omega(l,-k)\omega(-l,-k)(l,k)[1]_lQ\otimes e_{i+1}\\
  &+\sum_{l = 1}^{i}\prod_{n=1}^{i}\prod_{s=1,s\neq l}^i \prod _{t=i+1}^d \omega (i+1,n)\omega (l,s)\omega (t,-l) \omega (-t,-l) \omega(l,-l)[1]_lQ \otimes e_{i+1}\\
   &+ \sum_{m=1}^i\sum_{l = 1,l\neq m}^{i}\sum_{k = i+2}^{d} \prod_{n=1,n\neq m}^{i+1}\prod_{s=1,s\neq l,m}^{i+1} \prod _{t=i+2,t\neq k}^d \omega (m,n)\omega (l,s)\omega (t,k) \omega (-t,k) \omega(-k,k)\\
 &\quad\quad\quad \quad\quad \quad \quad\quad\quad\quad\quad \quad  \omega(l,k)\omega(-l,k)\omega(m,k)\omega(-m,k)(m,i+1)(k,l)Q\otimes e_{i+1}\\
   &+ \sum_{m=1}^i\sum_{l = 1,l\neq m}^{i} \prod_{n=1,n\neq m}^{i+1}\prod_{s=1,s\neq l,m}^{i+1} \prod _{t=i+2}^d \omega (m,n)\omega (l,s)\omega (t,m) \omega (-t,m) \omega(-m,m)\\
 &\quad\quad\quad \quad\quad \quad \quad\quad\quad\quad\quad \quad  \omega(l,m)\omega(-l,m)(l,i+1)Q\otimes e_{i+1}\\
  &+ \sum_{m=1}^i\sum_{k = i+2}^{d} \prod_{n=1,n\neq m}^{i+1}\prod_{s=1,s\neq m}^{i} \prod _{t=i+1,t\neq k}^d \omega (m,n)\omega (i+1,s)\omega (t,k) \omega (-t,k) \omega(-k,k)\\
 &\quad\quad\quad \quad\quad \quad \quad\quad\quad\quad\quad \quad  \omega(m,k)\omega(-m,k)(k,m)Q\otimes e_{i+1}\\
   &+ \sum_{m=1}^i \prod_{n=1,n\neq m}^{i+1}\prod_{s=1,s\neq m}^{i} \prod _{t=i+1}^d \omega (m,n)\omega (i+1,s)\omega (t,m) \omega (-t,m) \omega(-m,m)Q\otimes e_{i+1}\\
     &+ \sum_{m=1}^i\sum_{l = 1,l \neq m}^{i} \prod_{n=1,n\neq m}^{i+1}\prod_{s=1,s\neq l,m}^{i+1} \prod _{t=i+2}^d \omega (m,n)\omega (l,s)\omega (t,l) \omega (-t,l) \omega(-l,l)\\
 &\quad\quad\quad \quad\quad \quad \quad\quad\quad\quad\quad \quad  \omega(m,l)\omega(-m,l)(i+1,m)Q\otimes e_{i+1}\\
  &+ \sum_{m=1}^i \prod_{n=1,n\neq m}^{i+1}\prod_{s=1,s\neq m}^{i} \prod _{t=i+2}^d \omega (m,n)\omega (i+1,s)\omega (t,i+1) \omega (-t,i+1) \omega(-(i+1),i+1)\\
 &\quad\quad\quad \quad\quad \quad \quad\quad\quad\quad\quad \quad \omega(m,i+1)\omega(-m,i+1) (m,i+1)Q\otimes e_{i+1}\\
   &+ \sum_{m=1}^i\sum_{l = 1,l\neq m}^{i}\sum_{k = i+2}^{d} \prod_{n=1,n\neq m}^{i+1}\prod_{s=1,s\neq l,m}^{i+1} \prod _{t=i+2,t\neq k}^d \omega (m,n)\omega (l,s)\omega (t,-k) \omega (-t,-k) \omega(k,-k)\\
 &\quad\quad\quad \quad\quad \quad \quad\quad\quad  \omega(l,-k)\omega(-l,-k)\omega(m,-k)\omega(-m,-k)(m,i+1)(k,l)[1]_{l}Q\otimes e_{i+1}\\
   &+ \sum_{m=1}^i\sum_{l = 1,l\neq m}^{i} \prod_{n=1,n\neq m}^{i+1}\prod_{s=1,s\neq l,m}^{i+1} \prod _{t=i+2}^d \omega (m,n)\omega (l,s)\omega (t,-m) \omega (-t,-m) \omega(m,-m)\\
 &\quad\quad\quad \quad\quad \quad \quad\quad\quad\quad\quad \quad  \omega(l,-m)\omega(-l,-m)(l,i+1)[1]_mQ\otimes e_{i+1}\\
   &+ \sum_{m=1}^i\sum_{k = i+2}^{d} \prod_{n=1,n\neq m}^{i+1}\prod_{s=1,s\neq m}^{i} \prod _{t=i+1,t\neq k}^d \omega (m,n)\omega (i+1,s)\omega (t,-k) \omega (-t,-k) \omega(k,-k)\\
 &\quad\quad\quad \quad\quad \quad \quad\quad\quad\quad\quad \quad  \omega(m,-k)\omega(-m,-k)(k,m)[1]_mQ\otimes e_{i+1}\\
  &+ \sum_{m=1}^i \prod_{n=1,n\neq m}^{i+1}\prod_{s=1,s\neq m}^{i} \prod _{t=i+1}^d \omega (m,n)\omega (i+1,s)\omega (t,-m) \omega (-t,-m) \omega(m,-m)[1]_mQ\otimes e_{i+1}\\
       &+ \sum_{m=1}^i\sum_{l = 1,l \neq m}^{i} \prod_{n=1,n\neq m}^{i+1}\prod_{s=1,s\neq l,m}^{i+1} \prod _{t=i+2}^d \omega (m,n)\omega (l,s)\omega (t,-l) \omega (-t,-l) \omega(l,-l)\\
 &\quad\quad\quad \quad\quad \quad \quad\quad\quad\quad\quad \quad  \omega(m,-l)\omega(-m,-l)(i+1,m)[1]_lQ\otimes e_{i+1}\\
  &+ \sum_{m=1}^i \prod_{n=1,n\neq m}^{i+1}\prod_{s=1,s\neq m}^{i} \prod _{t=i+2}^d \omega (m,n)\omega (i+1,s)\omega (t,-(i+1)) \omega (-t,-(i+1)) \omega(i+1,-(i+1))\\
 &\quad\quad\quad \quad\quad \quad \quad\quad\quad\quad\quad \quad \omega(m,-(i+1))\omega(-m,-(i+1)) (m,i+1)[1]_mQ\otimes e_{i+1}
 \end{align*}
Then we have
 \begin{align*}
 &\frac{e_1^{2}f_{1}}{q+q^{-1}}(Q\otimes e_i)\\
 = &\sum_{l = 1}^{i}\sum_{k = i+1}^{d} \prod_{n=1,n\neq l}^{i}\prod_{s=1,s\neq l}^i \prod _{t=i+1,t\neq k}^d \omega (i+1,n)\omega (l,s)\omega (t,k) \omega (-t,k) \omega(-k,k)\\
 &\quad\quad\quad \quad\quad \quad \quad\quad\quad\quad\quad \quad \omega(l,k)\omega(-l,k)(l,k)Q\otimes e_{i+1}\\
 &+\sum_{l = 1}^{i}\prod_{n=1,n\neq l}^{i}\prod_{s=1,s\neq l}^i \prod _{t=i+1}^d \omega (i+1,n)\omega (l,s)\omega (t,l) \omega (-t,l) \omega(-l,l)Q \otimes e_{i+1}\\
 &+ \sum_{l = 1}^{i}\sum_{k = i+1}^{d} \prod_{n=1,n \neq l}^{i}\prod_{s=1,s\neq l}^i \prod _{t=i+1,t\neq k}^d \omega (i+1,n)\omega (l,s)\omega (t,-k) \omega (-t,-k) \omega(k,-k)\\
 &\quad\quad\quad \quad\quad \quad \quad\quad\quad\quad\quad \quad \quad \omega(l,-k)\omega(-l,-k)(l,k)[1]_lQ\otimes e_{i+1}\\
  &+\sum_{l = 1}^{i}\prod_{n=1,n \neq l}^{i}\prod_{s=1,s\neq l}^i \prod _{t=i+1}^d \omega (i+1,n)\omega (l,s)\omega (t,-l) \omega (-t,-l) \omega(l,-l)[1]_lQ \otimes e_{i+1}\\
   &+ \sum_{m=1}^i\sum_{l = 1,l< m}^{i}\sum_{k = i+2}^{d} \prod_{n=1,n\neq l,m}^{i+1}\prod_{s=1,s\neq l,m}^{i+1} \prod _{t=i+2,t\neq k}^d \omega (m,n)\omega (l,s)\omega (t,k) \omega (-t,k) \omega(-k,k)\\
 &\quad\quad\quad \quad\quad \quad \quad\quad\quad\quad\quad \quad  \omega(l,k)\omega(-l,k)\omega(m,k)\omega(-m,k)(m,i+1)(k,l)Q\otimes e_{i+1}\\
   &+ \sum_{m=1}^i\sum_{l = 1,l\neq m}^{i} \prod_{n=1,n\neq l,m}^{i+1}\prod_{s=1,s\neq l,m}^{i+1} \prod _{t=i+2}^d \omega (m,n)\omega (l,s)\omega (t,m) \omega (-t,m) \omega(-m,m)\\
 &\quad\quad\quad \quad\quad \quad \quad\quad\quad\quad\quad \quad  \omega(l,m)\omega(-l,m)(l,i+1)Q\otimes e_{i+1}\\
  &+ \sum_{m=1}^i\sum_{l = 1,l< m}^{i}\sum_{k = i+2}^{d} \prod_{n=1,n\neq l,m}^{i+1}\prod_{s=1,s\neq l,m}^{i+1} \prod _{t=i+2,t\neq k}^d \omega (m,n)\omega (l,s)\omega (t,-k) \omega (-t,-k) \omega(k,-k)\\
 &\quad\quad\quad \quad\quad \quad \quad\quad\quad  \omega(l,-k)\omega(-l,-k)\omega(m,-k)\omega(-m,-k)(m,i+1)(k,l)[1]_{l}Q\otimes e_{i+1}\\
  &+ \sum_{m=1}^i\sum_{l = 1,l\neq m}^{i} \prod_{n=1,n\neq l,m}^{i+1}\prod_{s=1,s\neq l,m}^{i+1} \prod _{t=i+2}^d \omega (m,n)\omega (l,s)\omega (t,-m) \omega (-t,-m) \omega(m,-m)\\
 &\quad\quad\quad \quad\quad \quad \quad\quad\quad\quad\quad \quad  \omega(l,-m)\omega(-l,-m)(l,i+1)[1]_mQ\otimes e_{i+1}.
 \end{align*}

Similarly, we have
 \begin{align*}
 &\frac{f_1e_1^{2}}{q+q^{-1}}(Q\otimes e_i)\\
 =&\sum_{k = i+2}^d\sum_{m = 1}^i \sum_{l=1,l >m}^{i+1}\prod_{t = i+2,t \neq k}^d\prod_{n =1, n \neq l,m}^{i+1}\prod_{s =1, s \neq l,m}^{i+1} \omega(t,k)\omega(-t,k)\omega(-k,k)\omega(m,n)\omega(l,s)\\
 &\quad\quad\quad \quad\quad \quad \quad\quad\quad\quad\quad \quad\quad\quad\quad \quad \omega(m,k)\omega(l,k)(l,i+1)(k,m)Q \otimes e_{i+1}\\
 &+\sum_{k = i+2}^d\sum_{m = 1}^{i+1} \prod_{t = i+2,t \neq k}^d\prod_{n =1, n \neq m}^{i+1}\prod_{s =1, s \neq m}^{i+1} \omega(t,k)\omega(-t,k)\omega(-k,k)\omega(m,n)\omega(k,s)\\
 &\quad\quad\quad \quad\quad \quad \quad\quad\quad\quad\quad \quad\quad\quad\quad \quad (m,i+1)Q \otimes e_{i+1}\\
 &+\sum_{k = i+2}^d\sum_{m = 1}^i \sum_{l=1,l >m}^{i+1}\prod_{t = i+2,t \neq k}^d\prod_{n =1, n \neq l,m}^{i+1}\prod_{s =1, s \neq l,m}^{i+1} \omega(t,-k)\omega(-t,-k)\omega(k,-k)\omega(m,n)\omega(l,s)\\
 &\quad\quad\quad \quad\quad \quad \quad\quad\quad\quad\quad \quad\quad\quad\quad \quad \omega(m,-k)\omega(l,-k)(k,m)(l,i+1)[1]_mQ \otimes e_{i+1}\\
  &+\sum_{k = i+2}^d\sum_{m = 1}^{i+1} \prod_{t = i+2,t \neq k}^d\prod_{n =1, n \neq m}^{i+1}\prod_{s =1, s \neq m}^{i+1} \omega(t,-k)\omega(-t,-k)\omega(k,-k)\omega(m,n)\omega(-k,s)\\
 &\quad\quad\quad \quad\quad \quad \quad\quad\quad\quad\quad \quad\quad\quad\quad \quad (m,i+1)Q \otimes e_{i+1}.
 \end{align*}
Similarly,
 \begin{align*}
&e_1f_1e_1(Q\otimes e_{i})\\
=&\sum_{l=1}^i \prod_{n=1}^i\prod_{t=i+2}^d\prod_{s=1,s \neq l}^i \omega(i+1,n)\omega(t,i+1) \omega(-t,i+1)\omega(-(i+1),i+1)\omega(l,s)\\
&\quad\quad\quad \quad\quad \quad \quad\quad\quad\quad\quad \quad\quad\quad\quad \quad  \omega(l,i+1)(l,i+1)Q \otimes e_{i+1}\\
& + \sum_{k= i+2}^d\sum_{l=1}^i \prod_{n=1}^i\prod_{t=i+1, t \neq k}^d\prod_{s=1,s \neq l}^i \omega(i+1,n)\omega(t,k) \omega(-t,k)\omega(-k,k)\omega(l,s)\\
&\quad\quad\quad \quad\quad \quad \quad\quad\quad\quad\quad \quad\quad\quad\quad \quad  \omega(l,k)(l,k)Q \otimes e_{i+1}\\
&+\sum_{k=i+1}^d \prod_{n=1}^i\prod_{t=i+1,t \neq k}^d\prod_{s=1,s \neq l}^i \omega(i+1,n)\omega(t,k) \omega(-t,k)\omega(-k,k)\omega(k,s)Q \otimes e_{i+1}\\
& + \sum_{m=1}^i\sum_{k= i+2}^d\sum_{l=1,l \neq m}^i \prod_{n=1, n \neq m}^{i+1}\prod_{t=i+2, t \neq k}^d\prod_{s=1,s \neq l,m}^i \omega(m,n)\omega(t,k) \omega(-t,k)\omega(-k,k)\omega(l,s)\\
&\quad\quad\quad \quad\quad \quad \quad\quad\quad\quad \omega(l,k)\omega(m,k)\omega(-m,k)\omega(l,i+1)(m,i+1)(l,k)Q \otimes e_{i+1}\\
& + \sum_{m=1}^i\sum_{k= i+2}^d \prod_{n=1, n \neq m}^{i+1}\prod_{t=i+2, t \neq k}^d\prod_{s=1,s \neq m}^i \omega(m,n)\omega(t,k) \omega(-t,k)\omega(-k,k)\omega(i+1,s)\\
&\quad\quad\quad \quad\quad \quad \quad\quad\quad\quad \omega(i+1,k)\omega(m,k)\omega(-m,k)(m,k)Q \otimes e_{i+1}\\
& + \sum_{m=1}^i\sum_{l= 1,l \neq m}^i \prod_{n=1, n \neq m}^{i+1}\prod_{t=i+2}^d\prod_{s=1,s \neq l,m}^i \omega(m,n)\omega(t,m) \omega(-t,m)\omega(-m,m)\omega(l,s)\\
&\quad\quad\quad \quad\quad \quad \quad\quad\quad\quad \omega(l,m)\omega(l,i+1)(l,i+1)Q \otimes e_{i+1}\\
& + \sum_{m=1}^i\prod_{n=1, n \neq m}^{i+1}\prod_{t=i+2}^d\prod_{s=1,s \neq m}^i \omega(m,n)\omega(t,m) \omega(-t,m)\omega(-m,m)\omega(i+1,s)\\
&\quad\quad\quad \quad\quad \quad \quad\quad\quad\quad \omega(i+1,m)Q \otimes e_{i+1}\\
& + \sum_{m=1}^i\sum_{k= i+2}^d \prod_{n=1, n \neq m}^{i+1}\prod_{t=i+2,t \neq k}^d\prod_{s=1,s \neq m}^i \omega(m,n)\omega(t,k) \omega(-t,k)\omega(-k,k)\omega(k,s)\\
&\quad\quad\quad \quad\quad \quad \quad\quad\quad\quad \omega(m,k)\omega(-m,k)\omega(k,i+1)(m,i+1)Q \otimes e_{i+1}\\
& + \sum_{m=1}^i\prod_{n=1, n \neq m}^{i+1}\prod_{t=i+2}^d\prod_{s=1,s \neq m}^i \omega(m,n)\omega(t,m) \omega(-t,m)\omega(-m,m)\omega(m,s)\\
&\quad\quad\quad \quad\quad \quad \quad\quad\quad\quad \omega(m,i+1)(m,i+1)Q \otimes e_{i+1}\\
& + \sum_{k = i+1}^d \sum_{l= 1}^i \prod_{n=1}^{i}\prod_{t=i+1,t \neq k}^d\prod_{s=1,s \neq l}^i \omega(i+1,n)\omega(t,-k) \omega(-t,-k)\omega(k,-k)\omega(l,s)\\
&\quad\quad\quad \quad\quad \quad \quad\quad\quad\quad \omega(l,-k)(l,k)[1]_lQ \otimes e_{i+1}\\
 & + \sum_{k = i+1}^d \prod_{n=1}^{i}\prod_{t=i+1,t \neq k}^d\prod_{s=1}^i \omega(i+1,n)\omega(t,-k) \omega(-t,-k)\omega(k,-k)\omega(-k,s)Q \otimes e_{i+1}\\
 & + \sum_{m = 1}^i \sum_{k = i+2}^d\sum_{l= 1, l \neq m}^i \prod_{n=1, n \neq m}^{i+1}\prod_{t=i+2,t \neq k}^d\prod_{s=1,s \neq l，m}^i \omega(m,n)\omega(t,-k) \omega(-t,-k)\omega(k,-k)\omega(l,s)\\
&\quad\quad\quad \quad\quad \quad \quad\quad \omega(l,-k)\omega(m,-k)\omega(-m,-k)\omega(l,i+1)(m,i+1)(k,l)[1]_lQ \otimes e_{i+1}\\
 & + \sum_{m = 1}^i \sum_{k = i+2}^d \prod_{n=1, n \neq m}^{i+1}\prod_{t=i+2,t \neq k}^d\prod_{s=1,s \neq m}^i \omega(m,n)\omega(t,-k) \omega(-t,-k)\omega(k,-k)\omega(i+1,s)\\
&\quad\quad\quad \quad\quad \quad \quad\quad \omega(i+1,-k)\omega(m,-k)\omega(-m,-k)(k,m)[1]_mQ \otimes e_{i+1}\\
& + \sum_{m = 1}^i \sum_{l= 1, l \neq m}^i \prod_{n=1, n \neq m}^{i+1}\prod_{t=i+2}^d\prod_{s=1,s \neq l，m}^i \omega(m,n)\omega(t,-m) \omega(-t,-m)\omega(m,-m)\omega(l,s)\\
&\quad\quad\quad \quad\quad \quad \quad\quad \omega(l,-m)\omega(l,i+1)(i+1,l)[1]_mQ \otimes e_{i+1}\\
& + \sum_{m = 1}^i  \prod_{n=1, n \neq m}^{i+1}\prod_{t=i+2}^d\prod_{s=1,s \neq m}^i \omega(m,n)\omega(t,-m) \omega(-t,-m)\omega(m,-m)\omega(i+1,s)\\
&\quad\quad\quad \quad\quad \quad \quad\quad \omega(i+1,-m)[1]_mQ \otimes e_{i+1}\\
& + \sum_{m = 1}^i  \sum_{k = i+2}^d \prod_{n=1, n \neq m}^{i+1}\prod_{t=i+2,t \neq k}^d\prod_{s=1,s \neq m}^i \omega(m,n)\omega(t,-k) \omega(-t,-k)\omega(k,-k)\omega(-k,s)\\
&\quad\quad\quad \quad\quad \quad \quad\quad \omega(m,-k)\omega(-m,-k)\omega(-k,i+1)(m,i+1)Q \otimes e_{i+1}\\
& + \sum_{m = 1}^i   \prod_{n=1, n \neq m}^{i+1}\prod_{t=i+2}^d\prod_{s=1,s \neq m}^i \omega(m,n)\omega(t,-m) \omega(-t,-m)\omega(m,-m)\omega(-m,s)\\
&\quad\quad\quad \quad\quad \quad \quad\quad \omega(-m,i+1)(m,i+1)Q \otimes e_{i+1}\\
 \end{align*}
Therefore, we have
 \begin{align*}
&\frac{e_1^2f_1 + f_1e_1^2}{q+q^{-1}} -e_1f_1e_1\\
=&A \prod_{n=1}^i\omega(i+1,n)Q\otimes e_{i+1} + \sum_{l=1}^iB_l \prod_{s=1,s\neq l}^{i+1}\omega(l,s)(l,i+1)Q \otimes e_{i+1} \\
&+ \sum_{l=1}^i\sum_{k=i+2}^d C_{kl}(k,l)Q\otimes e_{i+1} + \sum_{m = 1}^i\sum_{l = 1,l < m}^i\sum_{k=i+2}^dD_{klm}(m,i+1)(k,l)Q\otimes e_{i+1} \\
&+ \sum_{l=1}^i\sum_{k=i+2}^d E_{kl}(k,l)[1]_lQ \otimes e_{i+1} +  \sum_{m = 1}^i\sum_{l = 1,l < m}^i\sum_{k=i+2}^dF_{klm}(m,i+1)(k,l)Q\otimes e_{i+1}.
 \end{align*}
 Here
 \begin{align*}
A=&\sum_{l=1}^i\prod_{s=1,s\neq l}^i \prod_{t=i+2}^d\omega(l,s)\omega(t,l)\omega(-t,l)\omega(-l,l)\omega(-(i+1),l)\\
& +  \prod_{t=i+3}^d\prod_{s=1}^i\omega(t,i+2)\omega(-t,i+2)\omega(-(i+2),i+2)\omega(i+2,s)\\
& + \sum_{k=i+3}^d\prod_{t=i+2,t \neq k}^d \prod_{s=1}^i\omega(t,k)\omega(-t,k)\omega(-k,k)\omega(k,s)\\
& +  \prod_{t=i+3}^d\prod_{s=1}^i\omega(t,-(i+2))\omega(-t,-(i+2))\omega(i+2,-(i+2))\omega(-(i+2),s)\\
& + \sum_{k=i+3}^d\prod_{t=i+2,t \neq k}^d \prod_{s=1}^i\omega(t,-k)\omega(-t,-k)\omega(k,-k)\omega(-k,s)\\
& - \sum_{k=i+1}^d\prod_{t=i+1,t \neq k}^d \prod_{s=1}^i\omega(t,k)\omega(-t,k)\omega(-k,k)\omega(k,s)\\
& - \sum_{m=1}^i\prod_{t=i+2}^d \prod_{s=1,s\neq m}^{i+1}\omega(m,s)\omega(t,m)\omega(-t,m)\omega(-m,m)\\
& - \sum_{k=i+1}^d\prod_{t=i+1,t \neq k}^d \prod_{s=1}^i\omega(t,-k)\omega(-t,-k)\omega(k,-k)\omega(-k,s)
 \end{align*}
  \begin{align*}
B_l =& \prod_{n=1,n \neq l}^i\prod_{t=i+2}^d\omega(i+1,n)\omega(t,i+1)\omega(-t,i+1)\omega(-(i+1),i+1)\omega(-l,i+1)\\
&+\sum_{m=1,m \neq l}^i\prod_{n=1,n \neq l,m}^{i+1}\prod_{t=i+2}^d\omega(m,n)\omega(t,m)\omega(-t,m)\omega(-m,m)\omega(-l,m)\\
&+ \prod_{t= i+3}^d\prod_{n=1,n\neq l}^{i+1}\omega(t,i+2) \omega(-t,i+2)\omega(-(i+2),i+2)\omega(i+2,n)\\
&+\sum_{k = i+3}^d\prod_{t=i+2,t\neq k}^d\prod_{n=1, n \neq l}\omega(t,k)\omega(-t,k)\omega(-k,k)\omega(k,n)\\
&+\prod_{t=i+3}^d\prod_{n=1, n \neq l}\omega(t,-(i+2))\omega(-t,-(i+2))\omega(i+2,-(i+2))\omega(-(i+2),n)\\
&+\sum_{k=i+3}^d\prod_{t=i+2,t \neq k}^{d}\prod_{n=1,n \neq l}^{i+1}\omega(t,-k)\omega(-t,-k)\omega(k,-k)\omega(-k,n)\\
& - \prod_{n=1}^i\prod_{t=i+2}^d\omega(i+1,n)\omega(t,i+1)\omega(-t,i+1)\omega(-(i+1),i+1)\\
&- \sum_{m=1,m \neq l}^i\prod_{n=1,n \neq m}^{i+1}\prod_{t=i+2}^d\omega(m,n)\omega(t,m)\omega(-t,m)\omega(-m,m)\\
&- \sum_{k=i+2}^d\prod_{t=i+2,t \neq k}^{d}\prod_{n=1,n \neq l}^{i+1}\omega(t,k)\omega(-t,k)\omega(-k,k)\omega(-k,n)\omega(l,k)\omega(-l,k)\\
& - \prod_{t=i+2}^{d}\prod_{n=1,n \neq l}^{i+1}\omega(t,l)\omega(-t,l)\omega(-l,l)\omega(l,n)\\
& - \sum_{k=i+2}^d\prod_{t=i+2,t \neq k}^{d}\prod_{n=1,n \neq l}^{i+1}\omega(t,-k)\omega(-t,-k)\omega(k,-k)\omega(-k,n)\omega(l,-k)\omega(-l,-k)\\
& - \prod_{t=i+2}^{d}\prod_{n=1,n \neq l}^{i+1}\omega(t,-l)\omega(-t,-l)\omega(l,-l)\omega(-l,n)\\
 \end{align*}
 \begin{align*}
 C_{kl} = &\prod_{n = 1, n \neq l}^i\prod_{s=1,s \neq l}^i \prod_{t=i+1,t \neq k}^d\omega(i+1,n)\omega(l,s)\omega(t,k)\omega(-t,k)\omega(-k,k)\omega(l,k)\omega(-l,k)\\
 &+\prod_{n = 1, n \neq l}^i\prod_{s=1,s \neq l}^i \prod_{t=i+2,t \neq k}^d\omega(t,k)\omega(-t,k)\omega(-k,k)\omega(l,n)\omega(i+1,s)\omega(l,k)\omega(i+1,k)\\
  &-\prod_{n = 1}^i\prod_{s=1,s \neq l}^i \prod_{t=i+1,t \neq k}^d\omega(i+1,n)\omega(t,k)\omega(-t,k)\omega(-k,k)\omega(l,s)\omega(l,k)\\
 &- \prod_{n = 1, n \neq l}^{i+1}\prod_{s=1,s \neq l}^i \prod_{t=i+2,t \neq k}^d\omega(l,n)\omega(t,k)\omega(-t,k)\omega(-k,k)\omega(i+1,s)\omega(i+1,k)\omega(l,k)\omega(-l,k)
 \end{align*}
    \begin{align*}
 D_{klm} = &\prod_{n = 1, n \neq l,m}^{i+1}\prod_{s=1,s \neq l,m}^{i+1} \prod_{t=i+2,t \neq k}^d\omega(m,n)\omega(l,s)\omega(t,k)\omega(-t,k)\omega(-k,k)\omega(l,k)\omega(-l,k)\\
 & \quad\quad\quad \quad\quad \quad \quad\quad\quad\quad\quad \quad\quad \quad \quad\quad\omega(m,k)\omega(-m,k)\\
 &+\prod_{n = 1, n \neq l,m}^{i+1}\prod_{s=1,s \neq l,m}^{i+1} \prod_{t=i+2,t \neq k}^d\omega(t,k)\omega(-t,k)\omega(-k,k)\omega(l,n)\omega(m,s)\omega(l,k)\omega(m,k)\\
 &- \prod_{n = 1, n \neq m}^{i+1}\prod_{s=1,s \neq l,m}^{i} \prod_{t=i+2,t \neq k}^d\omega(m,n)\omega(t,k)\omega(-t,k)\omega(-k,k)\omega(l,s)\omega(l,k)\omega(m,k)\\
 & \quad\quad\quad \quad\quad \quad \quad\quad\quad\quad\quad \quad\quad \quad \quad\quad\omega(-m,k)\omega(l,i+1)\\
 &- \prod_{n = 1, n \neq l}^{i+1}\prod_{s=1,s \neq l,m}^{i} \prod_{t=i+2,t \neq k}^d\omega(l,n)\omega(t,k)\omega(-t,k)\omega(-k,k)\omega(m,s)\omega(m,k)\omega(l,k)\\
 & \quad\quad\quad \quad\quad \quad \quad\quad\quad\quad\quad \quad\quad \quad \quad\quad\omega(-l,k)\omega(m,i+1)\\
 \end{align*}
    \begin{align*}
 E_{kl} = &\prod_{n = 1, n \neq l}^i\prod_{s=1,s \neq l}^i \prod_{t=i+1,t \neq k}^d\omega(i+1,n)\omega(l,s)\omega(t,-k)\omega(-t,-k)\omega(k,-k)\omega(l,-k)\omega(-l,-k)\\
 &+\prod_{n = 1, n \neq l}^i\prod_{s=1,s \neq l}^i \prod_{t=i+2,t \neq k}^d\omega(t,-k)\omega(-t,-k)\omega(k,-k)\omega(l,n)\omega(i+1,s)\\
 &\quad\quad\quad \quad\quad \quad \quad\quad\quad\quad\quad \quad\quad \quad \quad\quad\omega(l,-k)\omega(i+1,-k)\\
   &-\prod_{n = 1}^i\prod_{s=1,s \neq l}^i \prod_{t=i+1,t \neq k}^d\omega(i+1,n)\omega(t,-k)\omega(-t,-k)\omega(k,-k)\omega(l,s)\omega(l,-k)\\
 &- \prod_{n = 1, n \neq l}^{i+1}\prod_{s=1,s \neq l}^i \prod_{t=i+2,t \neq k}^d\omega(l,n)\omega(t,-k)\omega(-t,-k)\omega(k,-k)\omega(i+1,s)\\
 &\quad\quad\quad \quad\quad \quad \quad\quad\quad\quad\quad \quad\quad \quad \quad\quad\omega(i+1,-k)\omega(l,-k)\omega(-l,-k)
 \end{align*}
     \begin{align*}
 F_{klm} = &\prod_{n = 1, n \neq l,m}^{i+1}\prod_{s=1,s \neq l,m}^{i+1} \prod_{t=i+2,t \neq k}^d\omega(m,n)\omega(l,s)\omega(t,-k)\omega(-t,-k)\omega(k,-k)\omega(l,-k)\\
 & \quad\quad\quad \quad\quad \quad \quad\quad\quad\quad\quad \quad\quad \quad \quad\quad \omega(-l,-k)\omega(m,-k)\omega(-m,-k)\\
 &+\prod_{n = 1, n \neq l,m}^{i+1}\prod_{s=1,s \neq l,m}^{i+1} \prod_{t=i+2,t \neq k}^d\omega(t,-k)\omega(-t,-k)\omega(k,-k)\omega(l,n)\omega(m,s)\omega(l,-k)\omega(m,-k)\\
 &- \prod_{n = 1, n \neq m}^{i+1}\prod_{s=1,s \neq l,m}^{i} \prod_{t=i+2,t \neq k}^d\omega(m,n)\omega(t,-k)\omega(-t,-k)\omega(k,-k)\omega(l,s)\omega(l,-k)\omega(m,-k)\\
 & \quad\quad\quad \quad\quad \quad \quad\quad\quad\quad\quad \quad\quad \quad \quad\quad\omega(-m,-k)\omega(l,i+1)\\
 &- \prod_{n = 1, n \neq l}^{i+1}\prod_{s=1,s \neq l,m}^{i} \prod_{t=i+2,t \neq k}^d\omega(l,n)\omega(t,-k)\omega(-t,-k)\omega(k,-k)\omega(m,s)\omega(m,-k)\omega(l,-k)\\
 & \quad\quad\quad \quad\quad \quad \quad\quad\quad\quad\quad \quad\quad \quad \quad\quad\omega(-l,-k)\omega(m,i+1)\\
 \end{align*}
 By a direct calculation,  we have
 $$A = B_l=-(q^{2d-3i-1} + q^{3i+1-2d}),\ C_{kl}=D_{klm} = E_{kl}=F_{klm}=0.$$
 That is,
 $$e_1^2f_1+f_1e_1^2 = (q+q^{-1})(e_1f_1e_1 - e_1(q^{-1}k_1^{-1}+qk_1)).$$
 Similarly, we can prove the last identity. Proposition follows.

\end{document}